\documentclass[article]{amsart}
\usepackage{amssymb,hyperref}
\usepackage{amsbsy}
\usepackage{amsmath}
\usepackage{amstext}
\usepackage{amsthm}   
\usepackage{amscd}
\usepackage{graphicx}
\usepackage{amssymb}
\usepackage{hyperref}\hypersetup{linktocpage, hidelinks}
\usepackage{tikz}

\usetikzlibrary{arrows}

\makeatletter
\newcommand*{\rom}[1]{\expandafter\@slowromancap\romannumeral #1@}
\makeatother

\newcommand\numberthis{\addtocounter{equation}{1}\tag{\theequation}}

\theoremstyle{plain}
\newtheorem{thm}{Theorem}[section]

\newtheorem{lem}[thm]{Lemma}
\newtheorem{prop}[thm]{Proposition}

\theoremstyle{definition}
\newtheorem{defn}[thm]{Definition}

\theoremstyle{remark}
\newtheorem{rem}[thm]{Remark}

\newtheorem{notation}[thm]{Notation}

\numberwithin{equation}{section}

\newcommand{\Qb}{{\mathbb Q}}

\newcommand{\Rb}{{\mathbb R}}
\newcommand{\Cb}{{\mathbb C}}
\newcommand{\Eb}{{\mathbb E}}

\newcommand{\Zb}{{\mathbb Z}}

\newcommand{\ib}{{\mathbf 1}}

\newcommand{\hb}{{\mathbf h}}
\newcommand{\Jbb}{{\mathbf J}}
\newcommand{\Ac}{{\mathcal A}}
\newcommand{\Bc}{{\mathcal B}}
\newcommand{\Bw}{{\widetilde{\Bc}}}

\newcommand{\Lc}{{\mathcal L}}
\newcommand{\Lt}{{\widetilde{\mathcal L}}}
\newcommand{\Pt}{{\widetilde{\mathcal P}}}
\newcommand{\Nt}{{\widetilde{\mathcal N}}}

\newcommand{\Oc}{{\mathcal O}}

\newcommand{\al}{{\alpha}}

\newcommand{\ra}{\rightarrow}

\newcommand{\twobytwo}[4]{\left[ \begin{matrix} #1 & #2 \\ #3 & #4
\end{matrix} \right]}

\newcommand{\hide}[1]{{}}

\newcommand{\rdot}{\!\cdot\!}

\theoremstyle{plain}
\newtheorem{thmL}{Theorem}
\begin{document}

\title[Dynamics of Ostrowski skew-product]{Dynamics of Ostrowski skew-product: \rom{1}. Limit laws and  Hausdorff dimensions
%\\ Limit laws  and Hausdorff dimensions for Ostrowski's   map  via transfer operators
}
\author{Val\'erie  Berth\'e}
\address{Université de Paris, CNRS, IRIF, F-75006 Paris, France}
\email{berthe@irif.fr}

\author{Jungwon Lee}
\address{LPSM, CNRS, Sorbonne Universit\'e, 4 Place Jussieu, 75005 Paris, France}
\email{jungwon@lpsm.paris; jungwon.lee@warwick.ac.uk}
\curraddr{Mathematics Institute, University of Warwick, Coventry CV4 7AL, UK}
\thanks{This work was supported by the Agence Nationale de la Recherche through the project CODYS (ANR-18-CE40-0007).}

\date{\today}

\begin{abstract}

We present a dynamical study of Ostrowski's map based on the use of transfer operators. The Ostrowski dynamical system is obtained as a skew-product of the Gauss map (it has the Gauss map as a base and intervals as  fibers) and produces expansions of real numbers with respect to an irrational base
given by  continued fractions. By studying spectral properties of the associated transfer operators, we show that the absolutely continuous invariant measure of the Ostrowski dynamical system has exponential mixing properties. We deduce a central limit theorem for random variables of an arithmetic nature, and motivated by applications in inhomogeneous Diophantine approximation, we also  get Bowen--Ruelle type implicit estimates in terms of spectral elements for the Hausdorff dimension of a bounded digit set. 

\end{abstract}

\maketitle

\begin{center}
{\em Dedicated to J\"org Thuswaldner  on the occasion of his  $50^{th}$ birthday}
\end{center}

%\setcounter{tocdepth}{1}
%\tableofcontents

\section{Introduction}

We are concerned with the Ostrowski transformation defined on $[0,1)^2$ by
\[ S(x,y)= \left( \{1/x \} , \{y/x\} \right)   \] 
where $\{ . \}$ stands for the fractional part. This is a simple skew-product extension of the Gauss map  defined on $[0,1)$ by $ x \mapsto \{1/x\}$, which is known  to describe dynamically  regular continued fraction expansions.  The Ostrowski map, introduced in \cite{Ito:86},  allows the  ergodic description of the behaviour  of  the digits  in the associated  numeration system (see Proposition \ref{ost} below).

Various interesting arithmetic, combinatorial  or dynamical aspects  of the Ostrowski map (and  numeration)  have been investigated in a wide range of applications from Diophantine approximation to  symbolic dynamics. Indeed, classical  discrepancy results about Kronecker sequences  rely  on the use of Ostrowski's numeration. It is also well-known that the Ostrowski map provides the  inhomogeneous best approximations of $y$ in base $x$, as  stressed  for instance in Ito--Nakada \cite{ItoNakada}, Berth\'e--Imbert  \cite{BertheImbert} or Beresnevich--Haynes--Velani \cite{BHV:20},  with the latter providing  estimates for sums of reciprocal of fractional parts motivated by applications in  multiplicative  Diophantine approximation. There  is  also a deep connection with Sturmian sequences, such as introduced by Morse--Heldund \cite{MH}, which arise in symbolic dynamics  from   two-letter codings  of  orbits under  irrational  rotations on the unit circle; see e.g., the survey \cite{Ber:01} and the references  therein.   We further refer to Bromberg--Ulcigrai \cite{BroUl:18} for the use of Ostrowski's map  as a renormalisation tool for temporal limit theorems on deterministic random walks and  to Arnoux--Fisher \cite{ArnFis:01} in connection with the geometry  of scenery flows of the  torus. Lastly, Hara--Ito \cite{HaraIto} showed that the Ostrowski numeration yields a characterisation  of real quadratic number fields  through  periodic expansions.

In this article, we discuss  dynamical aspects of the Ostrowski map in the context of  the study of  spectral properties of the associated  transfer operators. We observe a Gaussian behavior  for Birkhoff sums that gives a refinement of ergodic results from \cite{ItoNakada}, as well as Bowen--Ruelle type estimates  for the Hausdorff dimension of  fractal sets of   pairs of real numbers  defined as  having  bounded digits with respect to Ostrowski's numeration.

\subsection{Homogeneous and inhomogeneous Diophantine approximation}
We begin by recalling the general context and some elements of motivation in brief.  Denote by $\|\rdot \|$ the distance to the nearest integer. According to the classical result of Dirichlet, for any irrational number $x$, there exist infinitely many integers   $q>0$ such that  $q \cdot \|q x \| <1$.
 Moreover, it  is well-known that the best approximations $q$ are given by the convergents of  the continued fraction expansion of $x$. This naturally leads to the  notion
of badly approximable numbers   for which there exists   $\varepsilon >0$ such that 
$ q \cdot \| q x \| \geq  \varepsilon$ for any   $q$.

Homogeneous approximation is
   about the bounds on $\|q x \|$, and it  corresponds dynamically  to  the 
study of    the  orbit of  $ 0$ under the  action of  the irrational rotation on the unit  circle $\Rb / \Zb$.
 Basically, inhomogeneous Diophantine approximation considers the case of     the orbit   of any point $y$  by  shifting  the initial point $0$, i.e., given a point $y$,  it deals with the behavior of $\|q x-y \|$.

%Even though it seems to be a natural translation of the homogeneous theory, but in fact studying inhomogeneous analogues is not simple. Most of known results have been studied in the specific settings by using various techniques from metric number theory. 

In the inhomogeneous setting,  Minkowski \cite{Min:1900} first showed that for any irrational  number $x$ and for any  real  number $y$ which is not of the form $mx+n$, with $m$ and $n$ being integers, then 
\[ \liminf_{|q| \ra \infty} |q| \rdot \|q x -y\| \leq \frac{1}{4} .\]
Note that this bound has been improved  by  Khintchine \cite{Khin:1935} in the setting of one-sided approximation (i.e., when we consider positive   integers $q$)  to $1/ \sqrt{5}$,   for any real  number $y$. See also \cite{Cassels:54,CRS,Descombes,Sos:58} for  related   results  on the  sequence of inhomogeneous minima  based on the use of  Ostrowski's  numeration.

Thus, it is natural to  consider the inhomogeneous analogue of the notion of badly approximable numbers: for an irrational number $x$ and for  $\varepsilon>0$, set
\begin{equation} \label{bad:inhom:ord}
B(x, \varepsilon)= \left\{ y \in \Rb:  \liminf_{|q|\rightarrow \infty}  |q| \rdot  \|q x - y \|\geq \varepsilon \right\} .
\end{equation}
When $q$ is assumed to  be positive, it is  also simply possible to consider  one-sided approximations; see for instance  \cite{kim,CRS,BKLR:21}.

Some questions arising  from  Diophantine approximations, either  in the homogeneous or  in the  inhomogeneous case, are   typically related to the Hausdorff dimension estimate of the set of badly approximable numbers,  with  techniques from metric number theory or homogeneous dynamics.
 Recall  indeed that  $x$ is badly approximable if and only if the partial quotients of its continued fraction expansion are bounded. Thus the study of 
bounded type fractal sets arises in a natural way: for a given  integer $N \geq 2$,  denote by  $C_N$ the set of  real numbers  $x \in [0,1)$ which admit a continued fraction expansion $x=[0;a_1,a_2, \ldots]$ with all $a_i$ satisfying $ 1\leq a_i \leq N$. A powerful approach towards the study of  the Hausdorff dimension of $C_N$ is provided by the thermodynamic formalism via transfer operators  and dynamical determinants,  leading to  a Bowen--Ruelle type spectral characterisation for the Hausdorff dimension of $C_N$.  We refer  to Hensley \cite{hens2,hens1}, Cesaratto--Vall\'ee \cite{cesaratto:val} and Jenkinson--Gonzalez--Urba\'nski \cite{jenkinson}. We  also refer to Jenkinson--Pollicott \cite{jenkinson:pollicott} for an explicit numerical computation on the Hausdorff dimension of $C_N$ in the context of Zaremba's conjecture and to Das--Fishman--Simmons--Urba\'nski \cite{DFSU} for strengthening Hensley's  formula via small perturbations of a conformal iterated function system.

%The idea is to study the weighted transfer operator, dynamical determinant associated to the Gauss map acting on suitable function spaces. In particular, they could observe so-called a Bowen-Ruelle type formula for the Hausdorff dimension of $E_A$ in terms of the dominant eigenvalue of transfer operator. 

The situation for  inhomogeneous approximation is more contrasted; see for instance \cite{kim,BHKV:10,LSS:19}. In contrast with the homogeneous  case,   it  is proved in particular  by Bugeaud--Kim--Lim--Rams  \cite{BKLR:21}  that  the set  $B(x,\varepsilon)$ (defined in (\ref{bad:inhom:ord}))  has  full Hausdorff dimension for some positive $\varepsilon$  if and only if $x$ is singular on average, which is equivalent
to  the average of  the logarithms of the partial quotients  to tend to infinity.     See  also  recent works of Kim--Liao \cite{kim:liao} for uniform
inhomogeneous approximation or Bugeaud--Zhang \cite{bugeaud:zhang} for  the field of formal power series.

In the meanwhile, there has been no development of a spectral approach in the inhomogeneous setting.   We set up the first steps in this direction. Here, we introduce and study a concrete dynamical formulation  for a  certain  bounded-type set (see (\ref{set:hausdorff})  below), motivated  by  inhomogeneous Diophantine approximation that arises from the Ostrowski expansion (see also  the discussion at the end of \S\ref{sec:clt}). One  reason for this  specific  choice of  digits  is  that it allows an explicit   geometric description of   fundamental  digit sets  (i.e.,  sets of real numbers  having the same given  number  of  first  pairs of digits $(a, b)$ when applying the Ostrowski map, such as described in \S\ref{prep:dynamic})   in terms of   quadrangulars that will provide suitable   covers for the Hausdorff estimates (see  Figure  \ref{fig:3bis}).

\subsection{Set up and main results}
The Ostrowski transformation $S$ is  defined on $[0,1)^2$ by
$S(x,y)= \left( \{1/x \} , \{y/x\} \right) $. Let $(x,y)=(x_0,y_0)$ and $(x_i,y_i)=S^i(x,y)$ for all $i \geq 1$. We then get two sequences $(a_i)_i$ and $(b_i)_i$ of non-negative integers given by
$a_i=\lfloor \frac{1}{x_{i-1}} \rfloor$ and $b_i=\lfloor \frac{y_{i-1}}{x_{i-1}} \rfloor$ ($i\geq 1$). Note that the sequence $(a_i)_i$  provides the continued fraction expansion of $x=[0;a_1,a_2, \ldots]$. We denote by $\frac{p_i}{q_i}=[0;a_1,\ldots, a_i]$  the $i$-th convergent of $x$ and set $\theta_i :=q_i x - p_i$. Then the  sequence $(b_i)_i$ yields the  expansion $y= \sum_{i=1}^\infty b_i |\theta_{i-1}| $  for $y$ with respect to $x$.
  Moreover,  the  sequence of digits $(b_i)_i$  satisfies the admissibility conditions stated in next proposition
  which guarantee uniqueness of such an   expansion. We refer to Proposition \ref{ost} and \S\ref{ap:ost}  for more  precise details.  See also   Barat--Liardet \cite{BaratLiardet} or Berth\'e \cite{Ber:01} for further reading on  the Ostrowski numeration system.  In particular, a  related numeration system  is defined  for integers as a generalisation of the Zeckendorf representation  which involves Fibonacci numbers and the golden ratio.

\begin{prop}[Ostrowski numeration system] \label{intro:ost}
Let $x \in [0,1) \backslash \Qb$.  Let $(a_i)$ stand for the sequence of partial quotients in the continued fraction expansion of $x$. Every real number $y \in [0,1)$ can be written uniquely in the form 
\[ y= \sum_{i=1}^\infty b_i |\theta_{i-1}| \]
where $0 \leq b_i \leq a_i$ for all $i \geq 1$,  if $a_i=b_i$ for some $i$, then $b_{i+1}=0$, and $ a_i  \neq b_i$ for infinitely many odd and even  indices $i$.
\end{prop}

In view of  Proposition  \ref{intro:ost}, Ostrowski's numeration  can be  defined  with respect  either to  the basis $ (|\theta_{i}|)_i$  or  to the  basis $ (\theta_{i})_i$.
Accordingly,  there are  various  dynamical   systems  associated with  Ostrowski's numerations such as discussed  e.g. in  \cite{Ito:86,ItoNakada}, which   share a  strong  resemblance.
We have chosen to focus here on the simplest one which has moreover the particularity  of being formally close   at first sight to the  two-dimensional 
continued fraction  algorithm Jacobi--Perron algorithm, which is among the most famous  continued  fraction  algorithm  (see Remark \ref{rem:JP1} for more details on this comparison).

Proposition \ref{intro:ost} enables us to see that the Ostrowski transformation provides  inhomogeneous  approximations of $y= \sum_{i=1}^\infty b_i |\theta_{i-1}|=
\sum_{i=1}^\infty b_i (-1)^{i-1} \theta_{i-1}$ in base $x$. More precisely, setting 
\begin{equation}\label{eq:Mn}
 M_n:=\sum_{i=1}^n b_i  (-1)^{i-1}q_{i-1}, 
 \end{equation} 
the sequence $(M_n)_{n \geq 1}$ yields a sequence of approximants  such  that $M_n x$ converges to $y$ modulo $1$. See \S\ref{arith:iden}  for more details
and also \S\ref{sec:clt}.
 %More precisely, given an  irrational number $x$ in $(0,1)$, Ostrowski numeration allows   the expansion  of  integers  with respect to the numeration scale  $(q_n)$ as well as real numbers with respect to the basis $(q_n x-p_n)$, where $(p_n/q_n)$ stands for the   sequence of  convergents of $x$ in its   continued fraction expansion. For integers, Ostrowski numeration system is a generalization of the Zeckendorf representation (which involves Fibonacci numbers and the golden ratio). The Ostrowski map produces the digits for  expansions of real numbers. 

Motivated by the   study of badly approximable  numbers, we  consider the following bounded-type set of points whose Ostrowski expansion admits  restricted digits. Fix an integer $N \geq 2$,  and set
\begin{equation} \label{set:hausdorff}
E_N=\{ (x,y) \in [0,1)^2 : 1 \leq  a_i \leq N \ \mbox{and }   b_i=a_i-1 , \ \mbox{for all } i \geq 1 \}  .
\end{equation}
\hide{
Note that given a point $(x,y) \in E_N$, we have $y \in B_S(x,\varepsilon)$. Thus our inhomogeneous approximation that arises from the Ostrowski map can be understood by studying the Hausdorff dimension of $E_N$.}
We   study  the  Hausdorff dimension of $E_N$. We further ask how often the  digit condition of  (\ref{set:hausdorff}) happens. For instance, consider  for all $n \geq 1$ and $x=[0;a_1, a_2, \ldots , a_n , \ldots]$
\begin{equation} \label{par:diophantine}
D_{N,n}(x,y)=\#\{ 1 \leq i \leq n : a_i \leq N \ \mbox{and }   b_i =a_i-1  \} .
\end{equation}
Regarding $D_{N,n}$ as a random variable over  pairs of real numbers in $[0,1)^2$, we study its limiting behavior as $N$ goes to infinity. %In particular, our interest is when the probability space arises from special trajectories of $x$, for instance, $x$ is a quadratic irrational and $y \in \Qb(x)$. %ii) $x$ rational and $y \in \Zb[x]$.  
Notice that
$$
D_{N,n}(x,y) = \sum_{i=1}^n f (a_i, b_i)   =
 \sum_{i=1}^n f \left(\left\lfloor \frac{1}{x_{i-1}} \right\rfloor,  \left\lfloor \frac{y_{i-1}}{x_{i-1}} \right\rfloor \right)
$$
where $f$ is the membership function given by $f(a,b)=\ib_{\leq N}(a) \ib_{a-1}(b)$. %Precisely, $\Ib_N(x,y)=1$ if and only if $x \leq N$ and $y \geq x-\delta$. 
%Writing $(a_i, b_i)$ as 
%\[  \left(  \left\lfloor \frac{1}{S_1^{i-1}(x,y)} \right\rfloor,  \left\lfloor \frac{S_2^{i-1}(x,y)}{S_1^{i-1}(x,y)} \right\rfloor  \right) \]
%where $S_j=\pi_j S$ denotes the projection for $j=1,2$, 
Simply the quantity $D_{N,n}$ can be understood as a Birkhoff sum of the Ostrowski map. %with the observable $f$. 

%We also consider other observables that define interesting quantities related to the Ostrowski dynamical system.
% and call  the corresponding ergodic sums the Diophantine parameters. 

%We further study more general situation, that means, $\phi(a_i)$ with various cost functions\footnote{Adapting ideas of Bettin-Drappeau, no need to restrict our interests only to costs of moderate growth/not passing through strong probabilistic result due to Hwang.} rather than $a_i$. Costs  under study: Additive costs $\sum c(a_k,b_k)$;  $\log M_n= \log \sum  _{ i \leq n}  b_k q_{k-1}$; $ \log  q_n  | \beta -  \sum  _{ i \leq n}  b_k q_{k-1}|$. %Then it is reasonable to consider transfer operator. 

\bigskip 

Now we describe   our  approach, which  is  both dynamical  and functional, being based on classical  transfer operator techniques. For complex parameters $s,w$ and for some observable  $f : I^2 \rightarrow {\mathbb R}^+$, consider the weighted transfer operator associated to the Ostrowski map
\begin{equation} \label{int:op}
\Lc_{s,w}\phi(x,y)=  \sum_{(a,b) \atop 0 \leq b \leq a, 1 \leq a} \frac{\exp\left( w \rdot f \left(\frac{1}{a+x}, \frac{b+y}{a+x} \right)\right)}{(a+x)^{3s}} \phi\left(\frac{1}{a+x}, \frac{b+y}{a+x} \right) \cdot \ib_{S I_{a,b}}(x,y)  
\end{equation}
whenever the series converges (see \S\ref{sub:transferop}  for a precise assumption), where $I_{a,b}$ denotes    stands for the interior of the set of  points $(x,y)$ satisfying $ \lfloor 1/x \rfloor=a$ and $ \lfloor y/x \rfloor  =b$. In this article, we study   spectral properties of the operator for $(s,w)$ close to $(1,0)$, and for a  given function  $f : I^2 \rightarrow {\mathbb R}^+$ assumed to be of {\em moderate growth}, i.e.,  $f \circ \hb_{a,b}=O(\log a)$ for all $(a,b)\in \Ac$;  we  use the  fact  that the parameters $s$ and $w$ are related to the study of  probabilistic limit theorems for Birkhoff sums  and to Hausdorff dimensions, respectively.

In order to study the spectrum of this operator, the main difficulty is to deal with the characteristic functions $\ib_{S I_{a,b}}$ that appear in (\ref{int:op}), which precisely comes from the digit conditions from Proposition \ref{intro:ost}. However, for the Ostrowski map, the usual and well-known function spaces, e.g. $C^1$ (used in \cite{bv} for the study of  Gauss map),  or else the set of  H\"older functions, are not invariant under the action of  the transfer operator. We have chosen to adapt  the strategy   of Mayer \cite{mayer}   developed for a class of locally expanding maps of $[0,1]^d$,  to the current  Ostrowski   setting.
% due to the   combinatorial and arithmetical properties of the Ostrowski dynamical system. 
The operator is  then proved to be compact  which we plan to use  for further numerical  estimates (as e.g. in \cite{jenkinson:pollicott}).  We also follow Broise \cite{broise} which handles  the case of  the Jacobi--Perron  algorithm; this allows us   to strengthen  the comparison  between both algorithms. The  associated dynamical system   is   not a skew product  but a two-dimensional continued fraction algorithm,  however  it shares  common features with the Ostrowski's map (see Remark \ref{rem:JP1} for more details).

We  thus observe the existence of a coarse topological  partition\footnote{A \emph{topological partition of $[0,1)^2$} is a collection of pairwise disjoint open nonempty sets so that the union of their closures equals $[0,1]^2$.} for $[0,1)^2$ with the  two following  disjoint subsets $\Delta_0=\{(x,y): y<x\}$ and $\Delta_1=\{(x,y): x < y \}$. This allows us to interpret the behavior of the  transfer operator with respect to the multiplication by characteristic functions in terms of the admissibility between two partitions $I_{a,b}$ and $\Delta_i$ (see \S\ref{mod:mayer} for precise details). From this, we introduce a generalised transfer operator. For each $i \in \{0,1\}$, define 
\begin{equation} \label{intro:modop}
\Lt_{i,(s,w)} \Phi(x,y) = \sum_{j \in \{0,1\}} \sum_{(a,b) \in \Ac_{i,j}} \frac{\exp\left( w \rdot f \left(\frac{1}{a+x}, \frac{b+y}{a+x} \right)\right)}{(a+x)^{3s}} \cdot \Phi_j \left( \frac{1}{a+x}, \frac{b+y}{a+x}\right) 
\end{equation}
%where  \tcb{$\Ac_{i,j}$ is  index set such that for each $i$, the associated inverse branch $\hb_{a,b}$ is uniquely represented by an admissible pair $(j, \hb_{a,b}: \Delta_i \ra \Delta_j)$. }
where $\Phi=(\Phi_0, \Phi_1)$  with $\Phi_i$  acting on $\Delta_i$, and $\Ac_{i,j}$ is the set of pairs of  digits $(a,b)$ (with $a \geq 1$, $0 \leq b \leq a$)  such that for each $i$, the associated inverse branch $\hb_{a,b}$ maps $\Delta_i $ to $ \Delta_j$.

The generalised transfer operator is almost similar to $\Lc_{s,w}$ but  it will act on the function space with an additional index $i$ for the modified partition $\{ \Delta_i \}_{i \in \{0,1\}}$. More precisely, we consider    locally holomorphic functions   described in terms of the  space of  ``pairs''   of functions  that are   holomorphic on $\Omega \subseteq \Cb^2$,  where $\Omega $ is a bounded domain on which the inverse branches of the Ostrowski map can be analytically continued, defined as follows:
$$
\Bw(\Omega)=\{ \Phi=(\Phi_0, \Phi_1): \Omega \ra \Cb : \Phi_i \ \mbox{is bounded and holomorphic for } i=0,1  \}.
$$
%In other words, an element  $(u,v)$ in $\Omega$  is  mapped on $ (\Phi(0,u,v), \Phi(1,u,v))$ in ${\mathbb C}^2$. 
This is a Banach space when endowed  with the norm 
\[ \|\Phi \|=\sup_{i \in \{0,1\}} \sup_{(u,v) \in \Omega} | \Phi_i(u,v)|  \]
on which the generalised transfer operator $\Lt_{s,w}=(\Lt_{0,(s,w)},\Lt_{1,(s,w)})$ acts properly when $(s,w)$ is close to $(1,0)$ and  $f$ is assumed to be of moderate growth, i.e.,  $f \circ \hb_{a,b}=O(\log a)$ for all $(a,b)$ satisfying $1 \leq a$, $0 \leq b \leq a$. Then we first observe the following.

\begin{thmL}[Theorem \ref{spectrum1}]
The operator $\Lt_{s,w}$ on $\Bw(\Omega)$ is compact. In particular, it has a simple largest eigenvalue $\lambda_{s,w}$ whose modulus is strictly larger than all other eigenvalues.
\end{thmL}

A relation between $\Lt_{s,w}$ and $\Lc_{s,w}$ can be clearly seen. Set $(\kappa \Phi)(x,y):= \Phi_i(x,y)$ if $(x,y) \in \Delta_i$, then we get $\kappa(\Lt_{s,w} \Phi)=\Lc_{s,w}(\kappa \Phi)$. This specialisation shows that $\Lc_{s,w}$ has an eigenfunction $\kappa \Phi_{s,w}$ in  the  space $L^1([0,1]^2)$ with the same eigenvalue $\lambda_{s,w}$. Together with the density of periodic points in Proposition \ref{dense:periodic}, the spectral gap from Theorem A allows us to see that there exists an absolutely continuous invariant measure $\mu$ for the Ostrowski map, whose density is locally holomorphic (its explicit   expression, due to Ito \cite{Ito:86},  is given in Remark \ref{ito}). This further leads to:

\begin{thmL}[Proposition \ref{mixing} and Theorem \ref{clt}]
The Ostrowski dynamical system has exponential mixing with respect to $\mu$. Therefore, we have a central limit theorem for Birkhoff sums. For given $f$ under mild assumptions and for  $z \in \Rb$
\[ \mu \left\{ \frac{1}{\sqrt{n}} S_n f  \leq z   \right\} \longrightarrow \frac{1}{\sigma \sqrt{2 \pi}} \int_{-\infty}^z e^{-t^2/ 2 \sigma^2} dt  \]
for  a suitable constant $\sigma$ as $n$ goes to infinity.
\end{thmL}

Remark that this comes from the argument that the moment generating function of a Birkhoff sum $S_n f$ can be written in terms of transfer operators $\Lt_{1, \frac{it}{\sqrt{n}}}^n$ (following now the classical  approach initiated by Nagaev, see e.g. Broise \cite{Broise:96}). Thus specialising $f$, we observe a Gaussian behavior for  various Diophantine parameters.

Finally, we consider the operator $\Lc_{N,s}:=\Lc_{N,s,0}$, whose summand is constrained with digits $(a,b)$ satisfying the condition in (\ref{set:hausdorff}). Notice that the operator  $\Lc_{N,s} $ satisfies Theorem A with the largest eigenvalue $\lambda_{N,s}$. Together with   explicit estimates on the diameter of the partition $I_{a,b}$  (see Proposition \ref{prop:diameter}) and  a  bounded  distortion property of the Jacobian determinant (see Proposition \ref{distortion}), we have the following spectral description for the Hausdorff dimension of $E_N$.

\begin{thmL}[Proposition \ref{up:hausdorff} and \ref{low:hausdorff}]
Consider the restriction $S|_{E_N}$ and its  associated transfer operator $\Lc_{N,s}$. Then we have
\[ 3 s_1/ 2 \leq \dim_{H}(E_N) \leq \min \{s_2, 3s_1\} \]
for real variables $s_1, s_2$ satisfying $\lambda_{N,s_1}=1$ and $(N+1)^{2s_2} \lambda_{N,s_2}=1$.
\end{thmL}

When the upper and lower bound is the same, this type of implicit characterisation is  usually  called a Bowen--Ruelle formula. Bowen \cite{bowen} first showed that the Hausdorff dimension of a conformal repeller for quasicircles is equal to the zero of a pressure function. This study was  further developed by Ruelle \cite{ruelle} for more general hyperbolic settings in terms of transfer operators. For the continued fraction case, there have been analogous results in the setting  of  homogeneous Diophantine approximation (see e.g. \cite{cesaratto:val,hens1,hens2,jenkinson}).  

We remark that Theorem C is among  the first such examples in inhomogeneous approximation and also for  maps
related to multidimensional  continued fractions; note  that estimates for the Hausdorff dimension   of the Rauzy gasket  and for the Arnoux--Rauzy continued fractions    have been established in \cite{AHS:16,Fougeron,GM:20}.
However, the present  approach (i.e., the choice of coverings by  the  quadrangulars from Proposition \ref{prop:diameter})  is not sufficient for  obtaining  a more precise spectral description  in higher dimension with non-conformality. Note that the upper bound in Theorem C  is not expected to be   optimal; also the quantity  $s_1$  from Theorem C is related to an analogous   quantity for the Gauss map.
 We plan to refine such estimates in a  subsequent  paper.

\medskip

We outline the contents  of this paper as follows. In \S\ref{prep}, we discuss basic dynamical properties of the Ostrowski map and of its  transfer operator. In \S\ref{mod:ostrowski}, we introduce a finite  modified Markov partition and  its  associated generalised transfer operator. In \S\ref{sec:spectral}, we study the  spectral properties of these generalised  transfer operators acting on a suitable Banach space. This leads us to have a central limit theorem in \S\ref{sec:clt} and  elements of  description for the Hausdorff dimension of a fractal set  of  bounded digit  type   arising from the Ostrowski expansion in \S\ref{sec:diophantine}. In \S\ref{ap:ost}, we give a supplementary exposition to some arithmetic identities as well as  arithmetical  proofs.

\subsubsection*{Acknowledgements}

We would like to thank Fran\c{c}ois Ledrappier for his careful reading and   for  significantly improving the  first draft. We  thank Eda Cesarrato and Lingmin Liao for suggesting the refinements in  Proposition \ref{up:hausdorff} and Proposition \ref{low:hausdorff}, and also 
  Viviane Baladi, Charles Fougeron, Brigitte Vall\'ee  and Benedict Sewell for stimulating discussions and clarifying comments. 
Finally, we owe much to the referee for   his/her perceptive suggestions.

\section{The Ostrowski  dynamical system} \label{prep}

\subsection{Arithmetical identities} \label{arith:iden}

Let  $I^2=[0,1)^2$. 
The Ostrowski  map $S: I ^2\rightarrow I^2$ is given by  
\[ S(x,y)= \left( \{1/x \} , \{y/x\} \right)  \mbox{ for } x \neq 0, \quad S(0,y)= \left(0,0 \right)   .\] 
This is a skew-product extension of the Gauss map, which is defined on the unit interval $I=[0,1)$  by $ T:  x \mapsto \{1/x\}$. 
For  $(x,y) \in I^2$, set $(x_0,y_0):=(x,y)$ and $(x_i,y_i):=S^i(x,y)$ for all $i \geq 1$.  Then the sequence of  (pairs of) digits produced  by  the Ostrowski map is defined  as follows, for all  positive integer $i$:  
$$(a_i,b_i)= \left(\left\lfloor \frac{1}{x_{i-1}} \right\rfloor, \left\lfloor \frac{y_{i-1}}{x_{i-1}} \right\rfloor \right) \mbox{ 
if } x_{i-1}\neq 0,  \mbox{ otherwise }(a_i,b_i)=0.$$
We also recall that $(a_i)_{i}$ provides the  sequence of partial quotients in the continued fraction expansion of  $x=[0;a_1,a_2, \cdots]$, 
   $\frac{p_i}{q_i}=[0;a_1,\cdots, a_i]$ stands for  the $i$-th convergent of $x$,   and $\theta_i :=q_i x - p_i$ for all $i \geq 1$.

   We first  state some classical identities concerning  regular continued fractions. Let  $x \in [0,1)$.
   Then, for $ n \geq 0$, one has
\begin{equation} \label{eq:ybis}
x= \frac{p_n + T^n (x) p_{n-1}}{q_n + T^n (x) q_{ n-1}}
\end{equation} \label{eq:y}
 and 
\begin{equation} \label{eq:yter}
x T(x) \cdots  T^{n-1}(x)= \frac{1}{q_{n}+ T^{n}(x) q_{n-1} }=  | q_{n-1} x -p_{n-1}| =|\theta_{n-1}|.
\end{equation}
Note that   (\ref{eq:y})    and   (\ref{eq:yter})  can   easily be deduced  from the   classical matricial representation of continued fractions\footnote{An analogous matricial representation 
for the Ostrowski map $S$ is  given in Proposition \ref{prop:diameter}.}. Indeed, one has 
\begin{align*}
\left[\begin{array}{l}
x\\
1
\end{array}\right] &=x T(x) \cdots T^{n-1}(x)  \left [\begin{array}{ll}
0 & 1\\
1 & a_1 
\end{array}\right] \cdots \left [\begin{array}{ll}
0 & 1\\
1 & a_n 
\end{array}\right]   \begin{bmatrix}
T^{n}(x)\\
1
\end{bmatrix} \\
& =  x T(x) \cdots T^{n-1}(x)  \left [\begin{array}{ll}
p_{n-1} & p_{n}\\
q_{n-1}& q_{ n}
\end{array}\right]  \begin{bmatrix}
T^{n}(x)\\
1
\end{bmatrix}.
\end{align*}
%and considering the  equality given by the second row of the matrices on the right hand side gives
%the first identity in   (\ref{eq:yter}).

 For the action  of the Ostrowski map $S$ on the skew coordinate, note that we have $y_{i-1}=x_{i-1}(y_i+b_i)$  ($i \geq 1$),   and thus this yields for all positive $n$ that 
\[ y = x_0 x_1 \cdots x_{n-1} y_n +\sum_{i=1}^n x_0 x_1 \cdots x_{i-1} b_i . \]

Note that the series $\sum_{i \geq 1 }^{\infty}  b_i  |\theta_{i-1}|$ is convergent (with $0 \leq b_i \leq a_i$ for all $i$).
Indeed,   for all  positive $i$, one has 
$ b_i|\theta_{i-1}|=b_i |q_{i-1} x -p_{i-1}| \leq \frac{a_i} {q_{i}} \leq \frac{1}{q_{i-1}}$; we conclude by observing that $ q_i \geq  (\frac{1+ \sqrt 5}{2})^i$ for all $i$.

Together with the identity $x_0 x_1 \cdots x_i = (-1)^i \theta_i= | \theta_i|$ ($i\geq 0$), we deduce  that the sequence of digits  $(b_i)_{i}$
provides  the expansion of  $y$   in terms of the Ostrowski numeration system from Proposition \ref{intro:ost}.

Observe moreover  by (\ref{eq:yter}) that 
%$$q_n | y - \sum_{i=1}^n b_i |\theta_i||=q_ny_n  |\theta_n| \, x_0 \cdots x_n x_n^{-1}= q_ny_n|\theta_n| x_n^{-1}=
%\frac{|\beta_n|}{1 + \frac{q_{n-1}}{q_n}  \alpha_n}.$$
$$  0 \leq  y - \sum_{i=1}^n b_i |\theta_{i-1}|=y_n   x_0 \cdots  x_{n-1}   = y_n|\theta_{n-1}| =\frac{y_n}{q_{n} + q_{n-1} x_n} <1/q_n$$
since $ 0 \leq x_n, y_n \leq 1$.
%\frac{y_n}{1 + \frac{q_{n-1}}{q_n}  x_n} <1/2$$
Hence, by recalling  that  $\|\rdot \|$ stands for  the distance to the nearest integer, one 
has $\| y - \sum_{i=1}^n b_i |\theta_{i-1}| \| =y - \sum_{i=1}^n b_i |\theta_{i-1}|.$
Since  $\sum_{i=1}^n b_i |\theta_{i-1} |=\sum_{i=1}^n b_i  (-1) ^i \theta_{i-1} $  is congruent modulo 1 to  $( \sum_{i=1}^n b_i  (-1)^{i-1} q_i)x$, one 
gets  $$\|  y - (\sum_{i=1}^n b_i  (-1)^{i-1} q_i )x \|=y - \sum_{i=1}^n b_i |\theta_{i-1}|=  \sum_{i \geq n+1}^n b_i |\theta_{i-1}|.$$
Hence,    the sequence $(M_n)_n$,  with $M_n=\sum_{i=1}^n b_i q_{i-1}(-1)^{i-1}$ (such as defined in (\ref{eq:Mn})),    is such that 
$\| y- M_n x\| =  \sum_{i=n+1}^\infty b_i |\theta_{i-1}|   $.
Moreover, one has 
$|\theta_i|= |\theta_{i-2}| - a_i |\theta_{i-1}|$  for all $i \geq 1$.
One then deduces   by using telescoping sums  that for $n \geq 1$
\begin{equation} \label{eq:reste}
\sum_{i=n}^{\infty} a_i |\theta_{i-1}|= |\theta_{n-2}|+ |\theta_{n-1}| 
\end{equation} 
and thus 
 $\sum_{i=n+1}^{\infty} b_i |\theta_{i-1}| \leq  |\theta_{n-1}|+ |\theta_{n}|,$
 which implies that the sequence  $(M_nx)_n$ tends to $y$ modulo  $1$ at exponential speed.

We now   state a more complete version of Proposition \ref{intro:ost}, whose   proof  is given  below in \S\ref{ap:ost}.

\begin{prop}[Ostrowski expansions] \label{ost}
Let $x \in [0,1) \backslash \Qb$. Every real number $y \in [0,1)$ can be written  in the form 
\begin{equation} \label{expansion} 
y= \sum_{i=1}^\infty b_i |\theta_{i-1}| 
\end{equation}
%where for all $i \geq 1$, one has  $0 \leq b_i \leq a_i$, $b_{i+1}=0$ if $a_i=b_i$, and  $a_i \neq b_i$ for infinitely many  odd integers.
where the sequence $(b_i)_i$  is given by the Ostrowski map  $S$ applied to $(x,y)$. 

Moreover, for any $(x,y) \in [0,1)^2$,   the sequence   $(a_i,b_i)_i$ of  digits coming from the trajectories of the  Ostrowski map $S$  satisfies the following admissibility properties:
\begin{enumerate}
\item   \label{ost:1} $0 \leq b_i \leq a_i$ for all $i \geq 1$;
\item  \label{ost:2}  if $a_i=b_i$ for some $i $, then $b_{i+1}=0$;   \qquad \qquad (Markov condition) 
\item  \label {ost:3} $a_i \neq b_i$ for infinitely many  odd  and  even indices $i$.
\end{enumerate}
Moreover, an  expansion  of the form   (\ref{expansion}) is unique  provided  that the   sequence of   digits $(b_i)$  satisfies these admissibility  conditions. 

\end{prop}

\begin{rem} \label{rem:JP1}
The Ostrowski map is quite similar to  the 2-dimensional Jacobi--Perron continued fraction algorithm. The Jacobi--Perron map is  defined
on $I^2$ by $  \operatorname{T_{JP}}:  (x,y) \mapsto  \left( \{y/x\}, \{1/x \}  \right)$ if $x\neq 0$, and  $ (0,y) \mapsto  (0,0).$ 
It    produces a  multi-dimensional  continued fraction algorithm whereas  the Ostrowski map is    a skew-product of the Gauss map.
The digits produced  by    the Jacobi--Perron  algorithm satisfy   a  similar   Markovian condition.  
Indeed, if   $(a_n,b_n)= \left(\left\lfloor \frac{1}{x_{n-1}} \right\rfloor, \left\lfloor \frac{y_{n-1}}{x_{n-1}} \right\rfloor \right)$  with
$(x_n,y_n)= \operatorname{T_{JP}}^n (x,y)$ for all $n\geq 1$,  then one has for all $ n \geq 1$, $0 \leq b_n  \leq a_n$, $a_n  \geq 1$ and if $a_n =b_n$ then
$b_{n+1} \geq 1$.

For more on its  dynamical study and its  exponential convergence properties, see e.g.,  Lagarias  \cite{Lagarias:93}
 and  Broise \cite{broise}.
\end{rem}

\subsection{On Ostrowski's partition }  \label{prep:dynamic}

In this section, we study the Ostrowski map from a dynamical point of view. We introduce the    fundamental digit  partition  and discuss some basic properties.

%fundamental interval 
Let us    first decompose $I^2=[0,1)^2$ into the following countable   partition  given  by the    ``cylinders'', i.e., the   sets of   real numbers in $I^2$   having the same   first   pair of digits   $(a,b)$ when applying the Ostrowski map $S$. More precisely,  for any $x \in (0,1)$, there exists a unique positive integer $a$ such that $a \leq \frac{1}{x} < a+1$
(i.e.,  $a= \lfloor 1/x \rfloor $). This means that it gives the  first digit of  the continued fraction expansion of $x$, i.e.,  $x=[0;a, \cdots]$. Then for any $y \in [0,1)$,
there exists a  unique integer $b$ (i.e., $b= \lfloor y/x \rfloor  \geq 0$) such that $bx  \leq  y  <(b+1)x$. Here, we always  have  $b \leq a$ since $y<1$. Thus for any integers $a,b$ satisfying $ a\geq 1$ and $0 \leq b \leq a$, set 
\[ I_{a,b}:=\left\{ (x,y) \in I^2 : \frac{1}{a+1} < x < \frac{1}{a}, \  bx <  y< (b+1)x   \right\} . \]
The set  $I_{a,b}$ is the interior of the set of points $ (x,y) \in I^2 $ satisfying $  \lfloor 1/x \rfloor=a ,$  $ \lfloor y/x \rfloor  =b.$
Let us denote by $\Ac$ the set of  digits $(a,b)$, i.e.,  $$\Ac= \{(a,b)\in {\mathbb Z}^2:  1 \leq a,   \ 0 \leq b \leq a \}.$$ Then it follows that $\{ I_{a,b} \}_{(a,b) \in \Ac}$ forms a countable  topological  partition for $I^2$, called the   {\em  fundamental digit partition}. The atoms of this partition are called  {\em fundamental digit  sets}. 
They are  either  triangles or  quadrangulars. See Figure \ref{fig:1} for an illustration.

 \begin{figure}[h] \label{partitioncell}
   \centering
   \includegraphics[width=2.5 in]{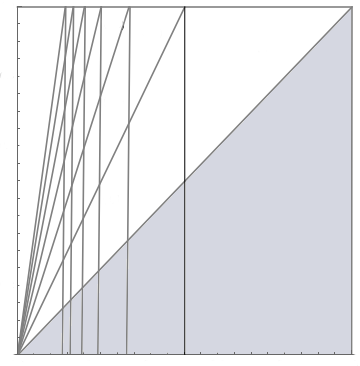} 
   \caption{Partition $I_{a,b}$ for $(a,b) \in \Ac$. The set  $\Delta_0$ is    depicted in light grey.}
   \label{fig:1}
\end{figure}
The  restriction of the  map $S$  on  each partition element $I_{a,b}$ is  one-to-one, hence it is  a bijection onto its image $S(I_{a,b})$ (given in Lemma \ref{markov} below). 
For any $(a,b) \in \Ac$, let  $\hb_{a,b}$ stand  for the inverse  branch of the restriction of $S$  to $I_{a,b}$. Note that it   has a  simple homographic  form:
\begin{equation} \label{inv:form}
 \hb_{a,b}(x,y) = \left(  \frac{1}{a+x}, \frac{b+y}{a+x}  \right)   . 
 \end{equation}
To describe  the image $S(I_{a,b})$, we divide the index set $\Ac$ into two disjoint subsets. Set $$\Ac_{0}:=\{(a,b)  \in \Ac: a=b \}  \mbox{ and }
\Ac_1:=  \{(a,b) \in \Ac :  0 \leq   b <a \} .$$
We may also divide    $I^2$ into two pieces, namely  
$$\Delta_0  = \left\{ (x,y) \in (0,1)^2 :y <x  \right\}  \mbox{ and  }   \Delta_1 = \left\{ (x,y) \in (0,1)^2 : x  < y   \right\} .$$
We  stress the fact that  the first   Ostrowski    pair of digits of the 
elements in $\Delta_0$   does not belong to  $\Ac_{0}$ as this notation might suggest.  However, we take this notation for providing   a  convenient expression of the Markov condition of Proposition \ref{ost} with respect to the inverse branch maps   $\hb_{a,b}$  such as  done  below in \S\ref{mod:mayer}.  

One has (by considering equality up to sets of zero measure)
 \begin{equation} \label{eq:markov}
 \Delta_0:=\bigcup_{a \geq 1} I_{(a,0)} \ \mbox{and  } \Delta_1:=I \backslash \Delta_0=  \bigcup_{(a,b)\in \Ac \atop   b \geq 1} I_{(a,b)} . 
\end{equation}

Now we can understand the admissibility properties from Proposition \ref{ost} in an alternatively way. Indeed,  next lemma  means that  the partition $\{I_{a,b}\}_{(a,b) \in  \Ac}$ is a  Markov partition  by   \eqref{eq:markov}.
This  also suggests us to express the associated transfer operator in an explicit way  taking into account  the  Markov condition (which will be done  below in  \S\ref{mod:mayer}).

\begin{lem} \label{markov}
The following holds for the Ostrowski map $S$:
\begin{align*}
&\mbox{ For  }(a,b) \in \Ac_0=\{(a,b)  \in \Ac: a=b \}, &S(I_{a,b})=\Delta_0. \\
&\mbox{ For }(a,b) \in \Ac_1=\{(a,b)  \in \Ac: 0  \leq b <a \}, &S(I_{a,b})=(0,1)^2. 
\end{align*}
\end{lem}

\begin{proof}
From the expression \eqref{inv:form}, it is clear  that $I_{a,b}=\hb_{a,b}((0,1)^2)$ for $(a,b) \in \Ac_1$ and $I_{a,b}=\hb_{a,b}((0,1)^2) \cap I^2=\hb_{a,b}(\Delta_0)$ for $(a,b) \in \Ac_0$. Since $S|_{I_{a,b}}^{-1}=\hb_{a,b}$ is bijective (being  an homography), the statement follows. 
\end{proof}

\hide{------

\begin{proof}
Let $(a,b) \in \Ac_0$  and let  $(x,y) \in (0,1)^2$. From $y/x<1/x$, we have $ \frac{1}{x}-a >\frac{y}{x}-a 
$ and  $S(x,y)= \left( \frac{1}{x}-a, \frac{y}{x}-a  \right) \in \Delta_0$.  Conversely, if $(x,y) \in  \Delta_0$, then 
$\hb_{a,a}(x,y) \in  I^2$. Hence $S(I_{a,a})=\Delta_0$.
Suppose  now that $(a,b) \in \Ac_1$.  Let  $(x,y) \in (0,1)^2$.  It suffices to check that $\hb_{a,b}(x,y) \in  (0,1)^2$ and this comes from  $0 <\frac{b+y}{a+x} \leq \frac{ a-1+y}{a+x} <1$.   
\end{proof}

---------}

\begin{rem}\label{rem:triangle}
Note that for  $(a,b)$ in $ \Ac_0,$ $I_{a,b}$ is a triangle and  for  $(a,b)$ in $ \Ac_1,$ $I_{a,b}$ is a quadrangular. Hence Lemma  \ref{markov} is consistent  with the fact 
 $\hb_{a,b}$  being  an homography, i.e.,  a triangular is sent  by $\hb_{a,b}$ onto a triangular, and similarly  a quadrangular is mapped onto a quadrangular.
 \end{rem}
We  will need the following   notation which corresponds to  refining the fundamental digit partition   by fixing a  finite number of  consecutive pairs of digits. 
\begin{notation}\label{not:h}
In all that follows, we   write $\Ac^n$ for the set of $n$-tuples of indexes $(a,b)=((a_1,b_1), \ldots, (a_n,b_n))$, and   for any $(a,b) \in \Ac^n$ $$\hb_{a,b}:=\hb_{a_1,b_1} \circ \cdots \circ \hb_{a_n,b_n}$$ stands  for the corresponding depth $n$ inverse branch.
We  also use the notation    (see e.g.  \eqref{op:iterates})
\begin{equation} \label{eq:hk}
\hb_{a,b}^{(k)}:= \hb_{a_k,b_k} \circ \cdots \circ \hb_{a_n,b_n},   \quad   1 \leq k \leq n.
\end{equation}
Now  define $$I_{a,b}:= \hb_{a,b}(0,1)^2 \cap (0,1)^2.$$
The topological   partition $\{I_{a,b}\}_{(a,b) \in \Ac^n}$ is called the  {\em partition by  fundamental sets  of depth $n$}. 
\end{notation}
We also  introduce  a  further notation that will allow the description in the next proposition  of   the fundamental sets in  the  case  of  bounded  pairs of digits  under study  in this paper. 
\begin{notation} \label{nota:In}
For $N \geq 2$,  let \[ \Ac_N=\{ (a,b)  \in \Ac: a \leq N  \mbox{ and }  b=a-1\}.\]
Let $ (a,b ) \in {\mathcal A}_N^n$. 
Set  $$I_{n}  (a_1, \cdots, a_n):= I_{a,b} \mbox{  for  }(a,b)= ((a_1,a_1-1), (a_2,a_2-1), \cdots, (a_n,a_n-1)).$$
We also  define $$A_n(a_1,\cdots,a_n):= \hb_{a,b}(0,0), \ B_n(a_1, \cdots, a_n):= \hb_{a,b}(0,1),$$
$$ C_n(a_1, \cdots,a_n):= \hb_{a,b}(1,1), \  D_n(a_1, \cdots, a_n):=\hb_{a,b}(1,0).$$
When there exists  no   risk of confusion, we use  the shortcut  $A_n$ for $A_n(a_1,\cdots,a_n)$,  with  the same holding  also   for  $B_n$, $C_n$, $D_n$.
%, as well as $A(0,0), B(0,1), C(1,0), D(1,1)$ for the vertices of the unit square. \tcb{Used?} 
\end{notation}

Now we   provide explicit  estimates  for the  measure and the  size
of fundamental sets  of depth $n$.
The proof of next proposition is given in \S \ref{ap:ost}.

\begin{prop} \label{prop:diameter}
Let $(a,b) \in \Ac ^n$.   Set  $q_{-1}=0$, $p_{-1}=1$, $q_0=1$, $p_0=0$, and 
for all  positive $k$ with  $ k \leq n-1 $,  let 
$q_{k +1}=a_{k +1} q_k +q_{k-1}$,   $p_{k +1}=a_{k +1} p_k +p_{k-1}$, and $\alpha_k=\sum_{i=1}^{k}    b_i | q_{k}p_{i-1}- p_k q_{i-1}|$.
%Let  $ \hb_{a,b}= \hb_{a_1,b_1} \circ \cdots \circ \hb_{a_n,b_n}$.
Then, the    matrix  of the homography  $ \hb_{a,b}$  is  equal to
\begin{equation} \label{mat:form}
\begin{bmatrix}
p_{n-1} & 0 & p_n\\
\alpha_{n-1} & 1 & \alpha_n\\
q_{n-1} & 0 & q_n
\end{bmatrix},
\end{equation}
\mbox{i.e., } 
$$
\hb_{a,b} (x,y)= \left(\frac{p_{n-1} x + p_n}{q_{n-1} x + q_n}, \frac{\alpha_{n-1} x +y+ \alpha_n}{q_{n-1} x + q_n}\right),$$ 
and  its 
 Jacobian determinant  $ \Jbb_{\hb_{a,b}} $  thus  satisfies for all $(x,y) \in I^2$:
$$|\Jbb_{\hb_{a,b}}(x,y)| = \left(\frac{1}{q_{n} +x q_{n-1}}\right)^3.$$
%\tcb{If $b_i \neq a_i$ for all $i$,} then   $I_{a,b}$ is  a quadrangular. We let denote its   four vertices    
%by   $A_n= \hb_{a,b}(0,0)$,  $B_n= \hb_{a,b}(0,1)$,  $C_n=\hb_{a,b}(1,0)$,  $D_n=\hb_{a,b}(1,1)$. 
%One has
%\begin{align*}
%A_n&=\left( \frac{p_n}{q_n}, \frac{\alpha_n}{q_n} \right), \ 
%B_n=\left( \frac{p_n}{q_n},  \frac{\alpha_n +1 }{q_n} \right),\\
%    C_n&=\left(\frac{p_n+p_{n-1}}{q_n+q_{n-1}},  \frac{\alpha_n+ \alpha_{n-1}+1}{q_n+q_{n-1}} \right), \ 
%D_n=  \left(\frac{p_n+p_{n-1}}{q_n+q_{n-1}},  \frac{\alpha_n+ \alpha_{n-1}}{q_n+q_{n-1}}  \right)  .
%\end{align*}
%Let $K_n$ stand  the maximum  of the absolute values of the  minors of order 2 of \eqref{mat:form}.
 Furthermore,  the   diameter    $\mathrm{diam}(I_{a,b})$ of  $I_{a,b}$ is in  $O(1/q_n)$. 
 % less than $\frac{ 4K_n}{q_n^2} $, where $K_n$ is the maximum  of the absolute values of the  minors of order 2 of \eqref{mat:form}.

  If moreover  $(a,b) \in  {\mathcal A}_N^n$, then $I_{a,b}$ is  a trapezium whose 
  (Lebesgue) measure    equals
$$ \mathrm{meas}(I_{a,b}) = \frac{q_{n-1}+2q_n}{2q_n^2 (q_n +q_{n-1})^2},$$  
 and  whose diameter  satisfies  $$ \frac{1}{q_n + q_{n-1}} \leq \mathrm{diam}(I_{a,b}) \leq  \frac{3}{q_n+q_{n-1}}.$$ 
 
 More precisely,   the trapezium $I_n (a_1, \cdots, a_n)$  is  inscribed  in  a  rectangle with  parallel  vertical sides      of  width $\frac{1}{q_n (q_n +q_{n-1})}$
   and  height in   $ \Theta(1/q_n)$.
   %;  hence its diameter is less than $\frac{2 K_n}{q_n^2}$, where $K_n$ is the maximum  of the absolute values of the  minors of order 2 of \eqref{mat:form}.
   If $n$ is odd (resp.  even),   then both  points $C_n=\hb_{a,b}(1,0)$ and  $D_n=\hb_{a,b}(1,1)$  are  located on a  vertical line      which is on the  left  (resp. right) of the vertical line   containing $A_n= \hb_{a,b}(0,0)$  and $B_n=\hb_{a,b}(0,1)$. 
\end{prop}

\begin{rem}It is a classical result about homographies that is   $h_A$ is a homography,  with $A \in  \mathrm{GL}(d + 1, {\mathbb R})$, then for all 
${\bf x} \in {\mathbb  R}^d $,
%  $\operatorname{Jac } h_A(x) = \frac{\det(A)}
%{(denominator)^{d+1}$  
  its Jacobian  at ${\bf x}$   can be expressed as
the quotient of the  determinant of $A$ (expressed at ${\bf x}$) by  the  denominator  of $h_A$ raised to the power  $d+1$  (see  e.g.,   \cite[Proposition 5.2]{Veech:78}).
For the Jacobi--Perron algorithm,   the expression of the Jacobian 
involves   denominators   where the variable  $y$ also occurs.  However, what is crucial in the Ostrowski  case is  that there is no dependence on the  $y$-coordinate, in constrast to the Jacobi--Perron case. 
\end{rem}

\hide{
\[ I_n(a,b)=\{ (x,y) \in I : x=[0;a_1, \cdots, a_n, S_1^n(x)]  \ \mbox{and } y= \sum_{i=1}^n b_i |\theta_{i-1}| + S_2^n(y) |\theta_{n-1}|  \}  \] 
where $S_j=\pi_j S$ is the projections for $j=1,2$. }

\subsection{Density of periodic points} \label{subsec:density} 

For the later purpose, we shall remark here one important aspect concerning the density of periodic points.

\begin{prop} \label{dense:periodic}
The periodic points of  the Ostrowski map $S$ are dense in $[0,1)^2$.
\end{prop}

\begin{proof}
We follow the idea of Broise \cite[Proposition 2.9]{broise} which is conducted for the case of  the Jacobi--Perron algorithm.
For a finite  sequence  of  digits $(a_1,b_1) \cdots   (a_n,b_n)$,  the notation $\overline{ (a_1,b_1) \cdots   (a_n,b_n)}$ 
stands for the infinite periodic  sequence with period  $(a_1,b_1) \cdots   (a_n,b_n)$.

Let $(x,y)$  be a point in $\bigcap _{ k \geq 0}   S^{-k}(I^2)$, i.e.,  $(x,y)$ admits an infinite expansion.   Let $(a_n,b_n)_{ n \geq1}$  stand for its  Ostrowski expansion.
Let us define a sequence  $(x^{(n)}, y^{(n)})_{ n \geq 1} $ of  elements in $I^2$    admitting  a periodic  Ostrowski expansion
 that converges to $(x,y)$. One can first check that for every positive integer $n$,
 the infinite sequence  $\overline{ (a_1,b_1) \cdots   (a_n,b_n) (1,0)}$   satisfies the admissibility conditions  in Proposition \ref{ost}.
Then we can take  $(x^{(n)}, y^{(n)})$ as an element of $I^2$ having   the periodic Ostrowski expansion 
$\overline{ (a_1,b_1) \cdots   (a_n,b_n) (1,0)}$.  
The point  $(x^{(n)},y^{(n)})$  is in the domain $ I_n(a_1,\cdots,a_n)$ (with the notation of  Proposition  \ref{prop:diameter}), and  this domain shrinks exponentially fast to $ (x,y)$ by Proposition \ref{prop:diameter}. We deduce that $(x^{(n)})$   converges to $x$ and  $(y^{(n)})$  converges to $y$.

 %The point $x^{(n)}$ has the 
%same   $n$ first  convergents as  $x$, which gives 
%$$ | x- x^{(n)}| \leq   \left| x- \frac{p_{n-1}}{q_{n-1}} \right|+ \left|x^{(n)}- \frac{p_{n-1}}{q_{n-1}} \right| \leq  \frac{2}{q_n+q_{n-1}}.$$

%We recall by (\ref{eq:reste}) that  $\sum_{k=n}^{\infty} b_k |\theta_{k-1}| \leq  |\theta_{n-2}|+ |\theta_{n-1}|.$
 %Hence, we have
%$$  \left| y- y^{(n)}\right| \leq \left|y - \sum_{ i =1}^{n-1}  b_i |\theta_{i-1}|\right| + \left|y^{(n)} - \sum_{ i =1}^{n-1}  b_i |\theta_{i-1}|\right|   \leq   2(\theta_{n-1} + \theta_n).$$

\end{proof}

\subsection{Transfer operator} \label{sub:transferop}

In this section, we introduce  a first transfer operator associated to the Ostrowski dynamical system.

Let $s, w \in \Cb$ be complex parameters and  let $f:I^2 \ra \Rb^+$ be a function. For $(x,y) \in I^2$ and $\phi \in L^1(I^2)$, consider the transfer operator 
\begin{align*}
\Lc_{s,w} \phi(x,y)=  \sum_{S(x',y')=(x,y)} \frac{\exp(w \rdot f(x',y'))}{|\Jbb_{S}(x',y')|^s} \cdot \phi(x',y')
\end{align*}
where $\Jbb_f(x,y)$ denotes the Jacobian determinant of $f$ at $(x,y)$. 
Notice that for any $(a,b) \in \Ac$, we have  (see also Proposition \ref{prop:diameter})
\[ \Jbb_{\hb_{a,b}}(x,y)= \det \left( \twobytwo{-\frac{1}{(a+x)^2}}{0}{-\frac{b+y}{(a+x)^2}}{\frac{1}{a+x}}   \right)= -\frac{1}{(a+x)^3}  \]
which is non-vanishing and uniform with regard to the skew coordinate. %and $|\Jbb_{\hb_{a,b}}(x,y)| \leq \frac{1}{a^3} \leq 1$. 
Then from \S\ref{prep:dynamic}, the operator can be written in a more explicit way as
\begin{align} \label{op:expression}
\Lc_{s,w} \phi(x,y) &=  \sum_{(a,b) \in \Ac} e^{w \cdot f \circ \hb_{a,b}(x,y)} \cdot |\Jbb_{\hb_{a,b}}(x,y)|^s  \cdot \phi\circ \hb_{a,b}(x,y) \cdot \ib_{S I_{a,b}}(x,y) \\
&= \sum_{(a,b) \in \Ac} \frac{e^{w \cdot f\left(\frac{1}{a+x}, \frac{b+y}{a+x} \right)}}{(a+x)^{3s}} \cdot \phi \left(\frac{1}{a+x}, \frac{b+y}{a+x}\right) \cdot \ib_{S I_{a,b}}(x,y)  .
\end{align}
This series converges when $\Re(s)>1/3$ and $\Re(w)$ is close to $0$,  by assuming  that $f \circ \hb_{a,b}=O(\log a)$ for all $(a,b)\in \Ac$. From now on, we make this assumption on $f$ throughout the entire paper.  According to the terminology from  \cite{bv}, such a function $f$ is called of {\em moderate growth}.

Remark that certain terms may  vanish depending on the admissibility between $\Delta_i$ and $I_{a,b}$. Indeed Lemma \ref{markov} says that if
 $(x,y) \in \Delta_0$, then the characteristic function $\ib_{S I_{a,b}}(x,y)=1$ for all $(a,b) \in \Ac$. However if $(x,y) \in \Delta_1$, there is no inverse image of $(x,y)$ under $S$ in $I_{a,b}$ for $(a,b) \in \Ac_0$ and thus we have $\ib_{S I_{a,b}}(x,y)=0$ for $(a,b) \in \Ac_0$.

%Hence the characteristic functions at the end describe the iteration of transfer operator in a combinatorial way, more precisely, it can be encoded by the admissible sequence on $\Lc_0$ and $\Lc_1$ depending on the compatibility among $\Delta_0$, $\Delta_1$, and index $(a,b)$. We explain the details in the following section adapting Mayer \cite{mayer}.

Now observe the iterations of $\Lc_{s,w}$. We recall that  $\hb_{a,b}^{(k)}$  stands for  $\hb_{a_k,b_k} \circ \cdots \circ \hb_{a_n,b_n}$  for  $1 \leq k \leq n$.
Then by the chain rule applied to the Jacobian determinant and   by additivity of the exponential term, we have,  for any $n \geq 1$:
\begin{align*} 
 \Lc_{s,w}^n \phi (x,y)&= \sum_{(a,b) \in \Ac^n} e^{\sum_{k=1}^n w \cdot f \circ \hb_{a,b}^{(k)}(x,y)}|\Jbb_{\hb_{a,b}}(x,y)|^s \cdot \phi \circ \hb_{a,b}(x,y) \numberthis \label{op:iterates} \\
& \ \ \ \ \ \ \  \quad \quad  \cdot \prod_{k=1}^{n-1} \ib_{S I_{a_k,b_k}}  (\hb_{a,b}^{(k+1)})(x,y)  \cdot \ib_{S I_{a_n,b_n}}(x,y) . 
\end{align*}

\hide{----

For $n=2$, it is simply written as
\begin{align*}
\Lc_{s,w}^2 \phi(x,y) &= \sum_{(a, b) \in \Ac} e^{w \cdot f \circ \hb_{a,b}(x,y)} |\Jbb_{\hb_{a,b}}(x,y)|^s  \rdot \Lc_{s,w} \phi \left( \frac{1}{a+x}, \frac{b+y}{a+x} \right) \cdot \ib_{S I_{a,b}}(x,y) \\
&= \sum_{(a_2, b_2) \in \Ac} e^{w \cdot f \circ \hb_{a_2,b_2}(x,y)} |\Jbb_{\hb_{a_2,b_2}}(x,y)|^s \sum_{(a_1, b_1) \in \Ac}  e^{w \cdot f \circ \hb_{a_1,b_1} \circ \hb_{a_2,b_2}(x,y)} \\ 
& \ \ \ \ \ \ \ \ \ \cdot \left|\Jbb_{\hb_{a_1,b_1}}\left( \frac{1}{a_2+x}, \frac{b_2+y}{a_2+x} \right)\right|^s
 \phi \left( \frac{1}{a_1+\frac{1}{a_2+x}}, \frac{b_1+\frac{b_2+y}{a_2+x}}{a_1+\frac{1}{a_2+x}} \right) \\ 
& \ \ \ \ \ \ \ \ \ \cdot \ib_{S I_{a_1,b_1}}\left( \frac{1}{a_2+x}, \frac{b_2+y}{a_2+x} \right) \ib_{S I_{a_2,b_2}}(x,y) .
\end{align*}

-------}

 This shows that one has to  look into the action of multiplication by characteristic functions for a suitable choice of the function space on which the operator may act. This will be the main discussion for the next section.

\begin{rem}\label{rem:int}
In the classical case  $(s,w)=(1,0)$, i.e., where $ \Lc_{s,w} $ is the Perron--Frobenius operator,
one recovers the 
relation  $$\int _{I^2} \Lc_{1,0}\phi \ dm= \int_{I^2} \phi \ dm $$ with respect to the Lebesgue measure $m$  on $I^2$.
One has  also  for any $\phi$ in $L_1(I^2)$ and  any $ \psi$ in $L_{\infty}( I^2)$
\[ \int_{I^2}  \phi(x) \psi  \circ S \ dm= \int _{I^2}   (\Lc_{1,0} \phi  )\psi \ dm. \]
\end{rem}

\section{Markov partition  and  transfer operator} \label{mod:ostrowski}

In the case  of a  dynamical system  which is piecewise smooth and not complete  (as in  our case by  Lemma \ref{markov}), the main problem is to find a suitable space of functions on which the transfer operator admits nice spectral properties. We recall  by  (\ref{op:expression}) that     characteristic functions  of the form   $S I_{a,b} $ appear in the expression of the transfer operator $\Lc_{s,w} \phi$
and thus, some well-known function spaces  are not invariant under its action.

We remark that there have been extensive progress on this issue over the years in  the more general context of (an)isotropic Banach spaces on which  transfer operators associated to piecewise expanding or hyperbolic maps admit good spectral properties (see the  book of Baladi \cite{baladi:book2} for the  related  literature).  However, we are here  in a simple case. Indeed,    the  Ostrowski map  admits a Markov partition by  Lemma  \ref{markov}       for which  simple combinatorial arguments due to Mayer \cite{mayer} can be suitably adapted. The idea is,   firstly, to   adapt  the  (countable) Markov partition  provided by the  sets $I_{a,b}$
by considering a (finite)  modified Markov  partition which  controls   the   Markov  admissibility conditions  in Proposition \ref{ost}. Secondly,  the idea is to introduce a generalised transfer operator acting on a certain Banach space of holomorphic functions which are cut by discontinuities at the boundaries of the modified Markov partition. Note that this strategy  was also successfully used by  Broise in  \cite{broise} for the study of  the Jacobi--Perron algorithm (see also  \cite[Example 3]{mayer}).

In this section, we show that the Ostrowski dynamical system admits a finite Markov partition (relative to the   fundamental digit partition  $\{I_{a,b} \}_{(a,b) \in \Ac}$) in the sense of Mayer and thus consider the associated generalised transfer operator. We now give all explicit details for the  corresponding modification of the  Markov partition in \S\ref{mod:mayer} and for 
the  transfer operator in \S\ref{subsec:mto}.

\subsection{Markov partition in the sense of Mayer} \label{mod:mayer}

In this section, we shortly recall the ideas of Mayer   from \cite{mayer} and show that there exists a finite partition  of  $I^2$ satisfying    admissibility 
conditions with  respect to the  digit partition     $\{I_{a,b}\}$.
We recall that the admissible    sequences of   pairs of digits produced by the Ostrowski map $S$  satisfy a simple Markov condition (see Proposition \ref{ost}).

%and that the elements of $I_{a,b}$ have the 
%same first  pair of  Ostrowki digits ${(a,b)}$.

More precisely, Mayer  introduced   in \cite{mayer} a generalised transfer operator associated to a dynamical system $(I^k, T)$, where $I^k \subseteq \Rb^k$ denotes the $k$-dimensional unit cube, for   $T$    being a piecewise expanding with a countable partition. He showed that there is an ad hoc function space on which the operator acts and admits good spectral properties. 
 Mayer considered the following modified notion for an irreducible Markov partition:

\begin{defn}[Mayer] \label{markov:mayer} Let $\{ O_i\}_{i \in \mathcal{I}}$ be a topological   partition by open sets  for $I^k$. Denote by $h_i=T|_{O_i}^{-1}$ the inverse branch.
Consider  now a    topological partition $\{W_\al\}_{\al \in \mathcal{J}}$  by open set   of $I^k$ such that
\begin{enumerate}
\item[(M1)] For any $i \in \mathcal{I}$ and $\al \in \mathcal{J}$, there is either a unique $\beta \in \mathcal{J}$ satisfying   $$h_i(W_\al) \subseteq W_\beta \cap O_i
\mbox{ or }  h_i(W_\al) \subseteq \Rb^k \backslash I^k.$$
\item[(M2)] Given $\al, \beta \in \mathcal{J}$, there exists a finite chain $h_{i_1}, \cdots, h_{i_n}$ and $\gamma_1, \cdots, \gamma_n \in \mathcal{J}$ such that for all $1 \leq r \leq n$, we have $$h_{i_r} \circ \cdots \circ h_{i_n}(W_\al) \subseteq W_{\gamma_r}  \cap O_{i_r}$$ with $\gamma_1=\beta$. 
\hide{If $|\mathcal{J}|=\infty$, then there exists an integer $m>0$ such that for any two allowed chains 
\[ h_{i_1} \circ \cdots \circ h_{i_n}(W_\al) \subseteq W_\beta \ \mbox{and } h_{i_1} \circ \cdots \circ h_{i_n}(W_{\al'}) \subseteq W_{\beta'} , \] 
we have $\beta=\beta'$ if $n \geq m$.}
\end{enumerate}
Then the partition $\{W_\al\}_{\al \in \mathcal{J}}$ is called a Markov partition relative to $\{O_i\}_{i \in \mathcal{I}}$.
\end{defn}

The motivation for introducing such  a   modified partition is to deal with characteristic functions in the expression of  the  transfer operator (\ref{op:iterates})   for    establishing the    existence   of an  absolutely continuous invariant measure for piecewise expanding smooth maps that are not complete. 
In our case, we observe the existence of a finite Markov partition for the Ostrowski dynamical system that  satisfies the admissibility conditions  in Proposition \ref{ost} with respect to the partition $\{I_{a,b}\}$, as stated below.
\begin{prop} \label{prop:Mayer}
The partition $\{ \Delta_0, \Delta_1 \}$ is a Markov partition relative to $\{I_{a,b} \}_{(a,b) \in \Ac}$ in the sense of Mayer. 

\end{prop}

\begin{proof}
We first 
recall  that (see also Fig. \ref{fig:1}) $$\Delta_0=\{(x,y): y<x\}= \bigcup _{(a,b) \in \Ac: b=0}  I_{a,b} \mbox{  and }  \Delta_1=\{(x,y): x < y \}=  \bigcup_{(a,b) :  0<b \leq a }  I_{a,b} $$
 (by considering equality up to sets of zero measure).  Then we deduce from Lemma \ref{markov} by  straightforward calculation the following:
\begin{itemize}
\item
for $ (a,b)  \in \Ac_0 = \{(a,b) : a=b\}$,   $\hb_{a,b}(\Delta_0)=I_{a,b},$ 
%&\hb_{a,b}(\Delta_1)\subseteq[0,1]\times \Rb_{> 1},\\
$\hb_{a,b}(\Delta_1)\cap  I^2= \emptyset$;
\item
for $(a,b)  \in \Ac_1=\{(a,b) : 0 \leq b < a\},$ $ \hb_{a,b}(\Delta_0) \subseteq  I_{a,b}, $ $ \hb_{a,b}(\Delta_1) \subseteq I_{a,b},$  and 
$\hb_{a,b}((0,1)^2)= I_{a,b}.$
\end{itemize}

This shows that     the partition   $\{ \Delta_0, \Delta_1 \}$ satisfies Condition (M1). 
We now  introduce  some notation for proving   (M2).
For a  pair  of digits $(a,b) \in \Ac$, we  introduce a transition matrix $A^{(a,b)}$ encoding the admissibility between $I_{a,b}$ and $\{\Delta_0, \Delta_1\}$, with the entry  
 $A_{j,i}^{(a,b)}$ detecting  the admissibility from $\Delta_i$ to $\Delta_j$ with respect to $ \hb_{a,b}$; 
hence, for each $(a,b) \in \Ac$ and $i,j \in \{0,1\}$,  we set 
\begin{align*}
 A_{j,i}^{(a,b)}&=1 \ \mbox{if } \hb_{a,b}(\Delta_i) \subseteq \Delta_j \cap I_{a,b} \\
 &=0  \ \mbox{if }\hb_{a,b}(\Delta_i) \subseteq \Rb^2 \backslash \Delta_j. 
\end{align*}
%We abbreviate $A_{j,i}^{(a,b)}$ to $A_{j,i}$ when the index $(a,b)$ is clearly specified. 
We also set  (which is well-defined by  (M1)) $$\tau_{a,b}(i)=j  \mbox{ if } A_{j,i}^{(a,b)}=1 \mbox{  for some }j  \mbox{  and }\tau_{a,b}(i)=i \mbox{ otherwise}.$$

Let us come to the proof  of (M2) for the coarse partition $I^2$ provided by $\Delta_0 \cup \Delta_1$. 
By Lemma \ref{markov}  we observe    that for  $(a,b) \in \Ac$ and $i,j \in \{0,1\}$, then 
$$A^{(a,a)}=\begin{bmatrix}
0 &0 \\
 1& 0 
 \end{bmatrix}, \quad  A^{(a,0)}=\begin{bmatrix}
1 &1 \\
 0& 0
 \end{bmatrix},  \quad  A^{(a,b)}=\begin{bmatrix}
0 &0 \\
 1& 1 
 \end{bmatrix} \mbox{ when } 0 < b <a.$$

We  start  with a  simple remark. 
Fix $i\in \{0,1\}$. If  there exists $j$ such that  $ A_{j,i}^{(a,b)}=1$, then $j=\tau_{a,b}(i)$ and $ A_{ \tau_{a,b}(i),i}^{(a,b)}=1$.
  If  there  is no such  $j$,   the column  of index $i$  in the matrix   $ A^{(a,b)}$  has  only zero entries, and in particular  $ A_{ \tau_{a,b}(i),i}^{(a,b)}=0.$
Thus, one has   $\hb_{a,b} (\Delta_i ) \subset \Delta_j$ if and only if  $j=\tau_{a,b}(i)$  and $A_{\tau_{a,b}(i),i}^{(a,b)}=1$; in other words, 
 $\hb_{a,b}(x,y) \in \Delta_{\tau_{a,b}(i)} \cap I_{a,b}$ for  $(x,y) \in \Delta_i$  if and only if $A_{ \tau_{a,b}(i),i}^{(a,b)}=1$.

 We then consider the   following   graph: its   vertices are   $\Delta_0$ and $\Delta_1$ and   there is  an edge from state $i$ to $j$ 
$\Delta_i$ to $\Delta_j$ if  $\hb_{a,b} (\Delta_i ) \subset \Delta_j$, i.e.,  $j=\tau_{a,b}(i)$  and $A_{\tau_{a,b}(i),i}^{(a,b)}=1$.  
It is depicted in  Fig.  \ref{fig:2} below.

  Assertion   (M2)  then  comes from the  fact that this graph is   strongly connected.
Indeed, given  $i,j \in \{0,1\}$, there exists a   path from state $i$ to state $j$, i.e., a sequence $(a_k, b_k)_{1 \leq k \leq n} \in \Ac^n$ such that 
$\tau_{a_1,b_1} \circ \cdots \circ \tau_{a_n,b_n}(i)=j$  and 
$$\hb_{a_r,b_r} \circ \cdots \circ \hb_{a_n,b_n} (\Delta_i)  \subset \Delta_{\tau_{a_r,b_r} \circ \cdots \circ \tau_{a_n,b_n}(i)}, \mbox{ for all  } 1  \leq r \leq n,$$
which is equivalent to
\begin{equation} \label{irred}
%A_{\tau_1(i),i}^{(a_1,b_1)} A_{ \tau_1 \circ \tau_2(i),\tau_1(i)}^{(a_2,b_2)}  \cdots  A_{ \tau_1 \circ \cdots \circ \tau_n(i), \tau_{1} \circ  \cdots \circ \tau_{n-1}(i)}^{(a_n,b_n)}=1,    
 A_{ \tau_{a_n,b_n}(i),i}^{(a_n,b_n)} A_{ \tau_{a_{n-1},b_{n-1}} \circ \tau_{a_n,b_n}(i),\tau_{a_n,b_n}(i)}^{(a_{n-1},b_{n-1})}  \cdots  A_{ \tau_{a_1,b_1} \circ \cdots \circ \tau_{a_n,b_n}(i),\tau_{a_2,b_2} \circ  \cdots \circ \tau_{a_n,b_n}(i)}^{(a_1,b_1)} =1.
\end{equation}
%by using  $\tau_k$ as  short for $\tau_{a_k,b_k}$.
% in an expression of the form
%$$A_{\tau_n(i),i}^{(a_n,b_n)} A_{ \tau_n \circ \tau_{n-1}(i),\tau_1(i)}^{(a_{n-1},b_{n-1})}  \cdots  A_{ \tau_1 \circ \cdots \circ \tau_n(i), \tau_{1} \circ  \cdots \circ \tau_{n-1}(i)}^{(a_1,b_1)}$$
%for $((a_1,b_1), \cdots,(a_n,b_n)) \in \Ac^n$.   

\end{proof}

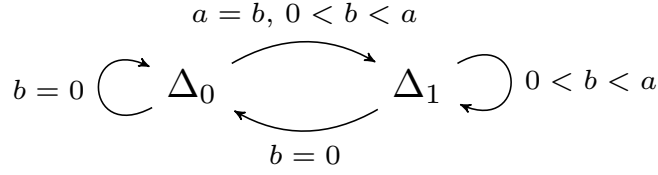
\begin{figure}[h]
    \centering
\begin{tikzpicture}[->,>=stealth',shorten >=1pt,auto,node distance=2cm,
                semithick,scale=1.5,transform shape]            
\tikzset{every loop/.style={min distance=7mm, looseness=5}}

\node   (A)                           {$\Delta_0$};
\node    (B) [right of=A]       {$\Delta_1$};

\path (A) edge [in=155, out=208, loop] node {\scriptsize $b=0$} (A)
          edge [bend left]  node {\scriptsize $a=b,$ $0<b<a$} (B)
      (B) edge [out=30, in=337, loop] node {\scriptsize $0<b<a$} (B)
          edge [bend left]  node {\scriptsize $b=0$} (A);
\end{tikzpicture}
 \caption{Admissibility among the  $\Delta_i$'s with respect to the action of $\hb_{a,b}$.  }
   \label{fig:2}
\end{figure}

\begin{rem}\label{rem:digits}

If  $\hb_{a,b} (\Delta_i ) \subset \Delta_j$ for some $j$, then the  first  pair  of Ostrowski  digit of  $\hb_{a,b}(x,y)$ is $(a,b)$, for   any $(x,y) \in \Delta_i$. 
Arrows  in the graph depicted in Fig. \ref{fig:2} then   indicate   what are the    admissible  pairs of  digits for  elements of $\Delta_0$ and $\Delta_1$   with respect to the  Markov condition of Proposition  \ref{ost}. 
\end{rem}

\subsection{Generalised transfer operator}\label{subsec:mto}

In this section, we define a generalised transfer operator associated to the Ostrowski map using the admissibility from \S\ref{mod:mayer} that  compares to the genuine one given  in (\ref{op:expression}),  which  involves  by (\ref{op:iterates})
$$\prod_{k=1}^{n-1} \ib_{S I_{a_k,b_k}} \circ (\hb_{a_{k+1},b_{k+1}} \circ \cdots \circ \hb_{a_n,b_n} )(x,y) \cdot \ib_{S I_{a_n,b_n}}(x,y).$$    

For each $i, j \in \{0,1\}$ (this index  refers  to  the  atoms of   the finite Markov partition $\{\Delta_i\}_{i \in \{0,1\}}$), set 
\[  \Ac_{i,j}=\{  (a,b) \in \Ac :     \hb_{a,b}({\Delta_i}   ) \subset \Delta_j  \}  .  \]
We also define  $\Ac_{i,j}^n$ as $ \Ac_{i,j}^n=\{  (a,b) \in \Ac^n:    \hb_{a,b}({\Delta_i}   ) \subset \Delta_j  \}$.  See Fig. \ref{fig:2}  where the edges from state $\Delta_i$ to 
state $\Delta_j$
 are labelled by elements of 
$ \Ac_{i,j}$.

For any $i$ and $(x,y) \in \Delta_i$, $\ib_{SI_{a,b}}(x,y)=1$  if and only if there exists $j$ such that  $(a,b) \in \Ac_{i,j}$. The   product of characteristic functions  in (\ref{op:iterates})  is  identified with 1  for $(x,y) \in \Delta_i$  if and only if  each term is equal to $1$,
which  is thus  in turn equivalent to the existence of $j$ such that   $((a_1,b_1), \cdots,(a_n,b_n)) \in  \Ac_{i,j}^n$.
%\[ \hb_{a_n,b_n} \circ \cdots \circ \hb_{a_1,b_1}: \Delta_i \ra \cdots \ra \Delta_j \in \Ac_{i,j}^n . \]

Using this, we introduce a generalised transfer operator. %For each $i, j \in \{0,1\}$ (this index  refers  to  the  atoms of   the modified Markov partition $\{\Delta_i\}_{i \in \{0,1\}}$), set 
%\[  \Ac_{i,j}=\{ \hb: \Delta_i \ra \Delta_j: \mbox{For some } (a,b) \in \Ac, \ \hb= \hb_{a,b}|_{\Delta_i}     \}    \]
%for the collection of all restricted inverse branch maps from $\Delta_i$ to $\Delta_j$.  
For each $i \in \{0,1\}$ and $(x,y) \in \Delta_i$, define $\Lt_{s,w}=(\Lt_{0,(s,w)}, \Lt_{1,(s,w)})$, where
\begin{align*} 
\Lt_{i,(s,w)} \Phi(x,y) &= \sum_{j \in \{0,1\}} \sum_{(a,b) \in \Ac_{i,j}}  e^{w \cdot f \circ \hb_{a,b}(x,y)}\cdot  |\Jbb_{\hb_{a,b}}(x,y)|^s \cdot \Phi_j \circ \hb_{a,b}(x,y)  \numberthis  \label{op:general} 
%&= \sum_{(a,b) \in \Ac} A_{\tau_{a,b}(i),i}^{(a,b)} \  \frac{e^{w \cdot f \circ \hb_{a,b}(x,y)}}{(a+x)^{3s}} \cdot \Phi \left(\tau_{a,b}(i), \frac{1}{a+x}, \frac{b+y}{a+x}\right).
\end{align*}
where $\Phi=(\Phi_0, \Phi_1)$.
Note that this operator is similar to $\Lc_{s,w}$, but it is  locally defined on each partition element $\Delta_0$ and $\Delta_1$   by means of the admissibility.

The explicit relation between the generalised operator $\Lt_{s,w}$ and $\Lc_{s,w}$ is not difficult to see. Say the operator $\Lt_{s,w}$ acts on $\widetilde{\Bc}$, a suitable function space which contains $L^1(I^2)$. Let $\kappa: \widetilde{\Bc} \ra L^1(I^2)$ be the specialisation given a.e.  by 
\[ (\kappa \Phi) (x,y):= \Phi_i(x,y) \ \mbox{if } (x,y) \in \Delta_i . \]
Then for $(x,y) \in \Delta_i$, we have
\begin{align*} 
\kappa (\Lt_{s,w} \Phi) (x,y) &= \Lt_{i,(s,w)} \Phi(x,y) \\
&= \sum_{j \in \{0,1\}} \sum_{(a,b) \in \Ac_{i,j}}   e^{w \cdot f \circ \hb_{a,b}(x,y)} \cdot  | \Jbb_{\hb_{a,b}}(x,y)|^s  \cdot \Phi_j \circ \hb_{a,b}(x,y) \\
&\overset{(*)}= \sum_{(a,b) \in \Ac} e^{w \cdot f \circ \hb_{a,b}(x,y)} |\Jbb_{\hb_{a,b}}(x,y)|^s  \rdot (\kappa \Phi) \circ \hb_{a,b}(x,y) \cdot \ib_{S I_{a,b}}(x,y) \\ 
&= \Lc_{s,w} (\kappa \Phi)(x,y)  \numberthis \label{kappa:special} .
\end{align*}
Here, we observe that $(*)$ holds since 
% $\hb_{a,b}(x,y) \in \Delta_{\tau_{a,b}(i)} \cap I_{a,b}$
for $(x,y) \in \Delta_i$, we have
$\ib_{SI_{a,b}}(x,y)=1$   if and only if $ (a,b) \in A_{i,j}$ for some $j$ (as stressed in the   beginning of this section). This implies in particular that both operators have the same spectral properties.

Mayer \cite{mayer} observed that there is a suitable function space for such piecewise expanding dynamical systems with a Mayer Markov partition on which the generalised transfer operator admits nice spectral properties. A simple but sensible idea is indeed  to cut by discontinuities, that is, to consider the space of  mappings into $\Cb \times \Cb$, indexed by $i \in \{0,1\}$, endowed with the sup-norm over both the domain and index. We will give a detailed exposition to this in the following section.

\bigskip

We finish this section with some miscellaneous remarks.  

\begin{rem}\label{rem:birk}
We see that the iteration of $\Lt_{s,w}$ can be written  for $n \geq 1$ as
\begin{align*}
\Lt_{i,(s,w)}^n \Phi(x,y) &= \sum_{j \in \{0,1\}} \sum_{(a,b) \in \Ac_{i,j}^n}  \exp \left( \sum_{k=1}^n w \rdot f \circ \hb_{a,b}^{(k)} ( x,y )\right)  \\
& \ \ \ \ \ \ \cdot |\Jbb_{\hb_{a,b}}(x,y)|^s \cdot \Phi_j \circ \hb_{a,b}(x,y)  \numberthis \label{iter:op} 
\end{align*}
where  $\hb_{a,b}^{(k)}$ is the partial composition of depth $1 \leq k \leq n$ of the inverse branch introduced in (\ref{eq:hk}). 
The summand is over the all admissible sequences  $(a,b) \in \Ac_{i,j}^n$, with respect to the Markov condition  of Proposition \ref{ost}.%   for   a given  pair $(x,y) \in \Delta_i$. 

Now, let $S_n f$ stand for  the Birkhoff sum of the  Ostrowski map $S$  for an observable $f$ of moderate growth   defined on $I^2$, i.e., 
$S_nf = \sum_{k =1}^n   f\circ S^k$. Let $(a,b)\in \Ac^n$. We observe that 
\[ S_n f ( \hb_{a,b}(x,y) )=   \sum_{k=1}^n f \circ   S ^k  (  \hb_{a,b}(x,y) )= \sum_{k=1}^n f \circ \hb_{a,b}^{(k)}(x,y) \]
since  for all $k$ with $  1 \leq k \leq n$ 
$$S^k  (  \hb_{a,b}(x,y) )= \hb_{a,b}^{(k)}(x,y).$$
Hence, in the exponential term 
\begin{equation*} \label{identity:ergodicsum}
\exp\left( \sum_{k=1}^n w \rdot f \circ \hb_{a,b}^{(k)} (x,y)\right) = \exp ( w \rdot S_n f (  \hb_{a,b}(x,y)))
\end{equation*} the parameter $w$ will be used for the study of probabilistic limit theorems below in \S\ref{sec:clt}. Further the parameter $s$ attached to the Jacobian determinant will play a role in the study of Hausdorff dimensions in \S\ref{sec:diophantine}. 
\end{rem}

\section{Spectrum of the generalised transfer operator} \label{sec:spectral}

Using the modification from \S\ref{mod:ostrowski}, we introduce and study an ad-hoc function space due to Mayer \cite{mayer} for the generalised transfer operator $\Lt_{s,w}$. We show that   $\Lt_{s,w}$ acts compactly on this function space and hence conclude that the Ostrowski dynamical system admits an absolutely continuous invariant measure which satisfies an  exponential mixing property.

Recall  that for $(a,b) \in \Ac_0$ and $(x,y) \in \Delta_1$,  $\hb_{a,b}(x,y)$ lies in $[0,1]\times \Rb_{>1}$. For such pairs, notice that $\frac{a+y}{a+x}>1$ and  that it  is close to 1 for $a \gg 1$. In the worst case, when $a=1$ with small $x$ and large $y$, we see that $\frac{a+y}{a+x}$ is bounded by 2. From this, we observe the existence of a complex domain on which $\hb_{a,b}$ admits 
an analytic continuation and that is mapped strictly into itself:

\begin{prop} \label{hol:domain}
There is a bounded domain $\Omega$ in $\Cb^2$ with $I^2 \subseteq \Omega$ on which the inverse branches $\hb_{a,b}$, for $(a,b) \in \Ac$,  can be analytically continued and   which map $\overline{\Omega} $ strictly into itself.
\end{prop}

\begin{proof}

We   first  consider   the  following  domain in $ {\mathbb R}^2$ which contains $I^2$: $$D:=\{ (x,y):  -1/4 < x  <   3/2, \    -  \alpha' x - \beta'  < y  < \alpha x + \beta\},$$
with   $ \beta > \alpha>1$, $ \beta' > \alpha'>0$ and  $\beta >1$ (We will see below that  a  suitable choice of parameters will be 
  $ \alpha =1.1$,  $\beta= 1.2$, $ \alpha' =1.03$, and  $  \beta'=1.05$).
Let $(x,y) \in D$ and let $(a,b) \in \Ac$. We recall that  $\hb_{a,b}(x,y) = (\frac{1}{a+x}, \frac{b+y}{a+x})$.
We want to prove that  the image of $D$  by  $\hb_{a,b} $ is strictly included into itself.

Consider  the first coordinate  $ \frac{1}{a+x} $ of  $\hb_{a,b}(x,y)$.
One has  $$0 <     \frac{1}{a+x}  <  \frac{1}{1-1/4} = 4/3< 3/2.$$
We consider now the second coordinate  $\frac{b+y}{a+x}$. Since  $ 0 \leq b \leq a$,
one gets  $$  \frac{- \alpha' x -\beta'}{a+x}   <  \frac{b+y}{a+x} < \frac{a + \alpha x +\beta}{a+x} .$$
We consider the upper  inequality which involves  $\beta  +  \frac{\alpha }{a+x}-   \frac{b+y}{a+x} $. One has 

$$\beta  +  \frac{\alpha }{a+x}-   \frac{b+y}{a+x} >
 \beta  +  \frac{\alpha}{a+x} -  \frac{a+\alpha x +\beta}{a+x} .$$
 We distinguish two cases according to the fact that $ a \geq 2$ or not.

 If $a \geq 2$,  one gets    $1- \frac{1}{a+x} > \frac{3}{ 7}$.  
  We then  use 
$$ \beta  +  \frac{\alpha}{a+x} -  \frac{a+\alpha x +\beta}{a+x}=    (\beta- \alpha)  \left( 1-\frac{1}{a+x}\right) + \frac{(\alpha-1) a}{ a+x} 
 >  (\beta- \alpha)  \left( 1-\frac{1}{a+x}\right),$$
 since $ \alpha>1$, to deduce  that 
$$\beta  +  \frac{\alpha }{a+x}-   \frac{b+y}{a+x} > 3/7 (\beta- \alpha) .$$
 If $a=1$,  $\frac{1}{1+x} > 2/5$ and 
 
$$  \beta  +  \frac{\alpha}{1+x} -  \frac{1+\alpha x +\beta}{1+x} =  \frac{(\beta- \alpha) x + \alpha-1}{ (1+x)} > \frac{2}{5} \left(-1/4(\beta-\alpha)  + \alpha-1\right)= 
 \frac{5 \alpha  -\beta-4} {10} $$
which holds for $5 \alpha > 4+\beta$ (which is the case  with the choice we made $\alpha=1.1$ and 
$ \beta= 1.2$).

For the lower inequality,  one has  
\begin{align*}
&  \frac{b+y}{a+x}  + \beta'  + \frac{\alpha' }{a+x} > \frac{- \alpha'  x -\beta'}{a+x} + \beta ' +\frac{\alpha' }{a+x} = \beta'- \alpha' +\frac{a \alpha' +\alpha'- \beta'  }{a+x} \\
& \geq  \beta'- \alpha' +\frac{2\alpha' - \beta'  }{a+x}    >  \beta'- \alpha' ,
  \end{align*}
since  we assume $ 2 \alpha' > \beta'$  (which is the case  with the choice we made 
$ \alpha' =1.03$ and  $  \beta'=1.05$). We thus have found a   domain $D$  in  the  real plane that is mapped  strictly  into inside  by all the  maps
  $\hb_{a,b}$.

For   defining a suitable domain $\Omega \subseteq \Cb^2$, we follow and refer to Broise \cite[Proposition 2.13]{broise}. We first projectivise the domain $D$ and then complexify it. Indeed, consider the  cone 
\[   \tilde{ D}= \{ (\lambda \rdot (x,y), \lambda): (x,y) \in D, \lambda \in \Rb \}     \]
where $\lambda \rdot (x,y)=(\lambda x, \lambda y)$ is the usual product by a scalar. 

The  matrix of the homography   $\hb_{a,b} $  is  the  matrix  $$M_{a,b}:=\begin{bmatrix} 
0 & 0 & 1\\
0 &1 & b\\
1 & 0 & a
\end{bmatrix}, $$ 
which  satisfies, for all $\lambda \in {\mathbb R}$, 
\[ M_{a,b}(\lambda \rdot (x,y), \lambda):= (\lambda x+ \lambda a) \rdot (\hb_{a,b}(x,y),1)  .  \]
By extending the linear  map  $M_{a,b}$ to the   complex  plane,  one deduces that the domain $\tilde{D}+ i \tilde{ D}$ is  mapped  strictly into itself.

Now set $\Omega:= \pi ( \tilde{D}+i \tilde{ D})$, where $\pi$ denotes the homogenisation $(z_1,z_2,z_3) \in \Cb^3 \mapsto (\frac{z_1}{z_3},\frac{z_2}{z_3}) \in \Cb^2$. Since $\hb_{a,b}$ is homographic, it has a natural holomorphic extension  $\hb_{a,b}=  \pi \circ  M_{a,b}$  to $\Omega$.
Hence we have   proved the existence of  a bounded domain $I ^2\subset \Omega \subseteq \Cb^2$ such that $\hb_{a,b}(\overline{\Omega}) $  is strictly included in $\Omega$.
\end{proof}

\begin{rem}
This gives another proof of the convergence of the Ostrowski algorithm  (see Proposition \ref{ost}) via the the use of the Hilbert  metric,    by following, e.g., the  same arguments as  in  \cite[Proposition 2.8]{broise}. This also yields that the Ostrowski map is  expanding.
\end{rem}

Here we remark that as a complex derivative $\Jbb_{\hb_{a,b}}(u,v)=-\frac{1}{(a+u)^3} \neq 0$ for all $(u,v) \in \Omega$, and this is equal to the Jacobian of the corresponding 4-dimensional real function up to a constant. Thus the series $\sum_{(a,b) \in \Ac} |\Jbb_{\hb_{a,b}}(u,v)|^s$ converges uniformly on $\Omega$ when $s$ is  close to  $1$. For $(a,b) \in \Ac^n$, we further have,  by Proposition \ref{prop:diameter}, the existence   of a constant $\eta<1$ so that  
$
|\Jbb_{\hb_{a,b}}| \ll \eta^n ,
$  for all  $(a,b) \in \Ac^n$.

Now consider the space $\Bw(\Omega)$    defined in the introduction as 
$$
\Bw(\Omega)=\{ \Phi=(\Phi_0,\Phi_1): \Omega \ra \Cb : \Phi_i  \ \mbox{is bounded and holomorphic for } i=0,1  \}.
$$
%of  piecewise holomorphic functions on $\Omega$ into $\Cb \times \Cb$ which is given by
%$$
%\Bw(\Omega)=\{ \Phi: \Omega \ra \Cb \times \Cb  :  \Phi \ \mbox{is bounded and holomorphic}   \},
%$$
%where $\Phi$ sends $(u,v) \mapsto (\Phi(0,u,v), \Phi(1,u,v))$.
This is a Banach space endowed with the norm 
\[ \|\Phi \|=\sup_{i \in \{0,1\}} \sup_{(u,v) \in \Omega} | \Phi_i(u,v)| . \]
Notice that the operator $\Lt_{s,w}$ acts properly on $\Bw$. For $\Phi \in \Bw(\Omega)$ and  $(s,w)$ close to $(1,0)$, we have
\begin{align*}
| \Lt_{i,(s,w)} \Phi(u,v)| & \leq \sum_{j \in \{0,1\}} \sum_{(a,b) \in \Ac_{i,j}} \left| \frac{e^{w \cdot f \circ \hb_{a,b}(u,v)}}{(a+u)^{3 \sigma}} \right| \rdot |\Phi_j \circ \hb_{a,b}(u,v)| \numberthis \label{bounded}
\end{align*}
where $\sigma=\Re(s)$.  Since $f$ is  assumed to be of  moderate growth   (i.e., $f \circ \hb_{a,b}=O(\log a)$ for all $(a,b) \in \Ac$), the series $\sum_{\hb \in \Ac_{i,j}} \left| \frac{e^{w \cdot f \circ \hb_{a,b}(u,v)}}{(a+u)^{3 \sigma}} \right|$ again converges and 
there  exists some constant ${M_{ \sigma}}$ which gives the boundedness  relation  
\begin{equation}  \label{eq:msigma} \| \Lt_{s,w} \Phi \| \leq {M_{ \sigma}} \| \Phi \|
\end{equation} by taking the supremum on both sides.

Following the main argument of Mayer, we study the spectrum of  the weighted generalised transfer operator $\Lt_{s,w}$ first for $(s,w)=(1,0)$. We observe:

\begin{thm} \label{spectrum0} 
Let $\Lt=\Lt_{1,0}$.
\begin{enumerate}
\item The operator $\Lt$ on $\Bw(\Omega)$ is compact.
\item Further, there is a spectral gap. There is a positive eigenvalue $\lambda_{1,0}$ whose modulus is strictly larger than all other eigenvalues and moreover   the  corresponding eigenfunction $\Phi_{1,0}$ is positive.
\end{enumerate}
\end{thm}

Here the positivity is given with respect to the real cone of functions taking positive values on $\Omega \cap \Rb^2$. We remark that our modified Markov partition $\{ \Delta_i \}_{i \in \{0,1\}}$   is finite, which simplifies a few arguments. The proof is based on  standard techniques from functional analysis: compactness follows from Montel's Theorem and the spectral gap is essentially due to the density of periodic points described in Proposition \ref{dense:periodic} for our case.
For more details,   see \cite{broise,mayer}.

\begin{proof}
First we claim (1). Pick a sequence $(\Phi_n)_{n \geq 1}$ in $\Bw(\Omega)$ with $\| \Phi_n \| \leq 1$. Then it is sufficient to show that $(\Lt \Phi_n)_{n \geq 1}$ has a convergent subsequence. Note that for each $i$ and $(u,v) \in \Omega$, by the above (\ref{bounded}) and (\ref{eq:msigma}), there exists $M_{i} >0$  (by setting  $\Lt_i=\Lt_{i,(1,0)}$)
\[ | \Lt_i \Phi_n (u,v)| \leq {M_{i} }\| \Phi_n \| , \]   hence $ \| \Lt \|$
is bounded by {$M_{i}$}  uniformly on $\Omega$  since $\| \Phi_n \| \leq 1$. This also  implies that $(\Phi_n)_{n \geq 1}$ satisfies the Montel property. Since the index set $\{0,1\}$  is finite, we can extract a subsequence $(\Phi_{n_k})_{k \geq 1}$ that converges uniformly to $\Phi$ on every compact subset of $\Omega$.  This yields the compactness of $\Lt$. %Further with this finite index set for modified partition, the operator can be extended as a nuclear of order zero, i.e., trace class operator. We will discuss the details in the following subsection.

We now explain (2). We rely on Perron--Frobenius theory by  Krasnoselski\u{\i} \cite{Kras}. Consider $\Bw(\Omega)$ as $ \Bw(\Omega \cap \Rb^2)+ \sqrt{-1} \cdot \Bw(\Omega \cap \Rb^2)$. (In order to avoid any confusion,  we
do not use $i$   but  $\sqrt{-1}$ in this proof). Let $K$ be the positive real cone 
\[ K=  \{ \Phi \in \Bw(\Omega \cap \Rb^2): \Phi_i(x,y) \geq 0 \ \mbox{for all } i\in \{0,1\}, (x,y) \in \Omega \cap \Rb^2 \}.  \]
Then it suffices to show that for any non-trivial $\Phi \in K$, we have $\Lt^n \Phi \in \mathrm{int}(K)$ for some $n \geq 1$. We claim that if we assume the contraposition, then Proposition \ref{dense:periodic} and Proposition \ref{prop:Mayer} yield the conclusion as follows. Suppose  indeed  that for any $n$, there are $i_n$ and $(x_n,y_n)$
 such that  $\Lt_{i_n}^n \Phi(x_n,y_n)=0$. This gives,  by recalling  (\ref{iter:op}):
\begin{align*}
\Lt_{i_n}^n \Phi(x_n,y_n) = & \sum_{j \in \{0,1\}} \sum_{(a,b)\in \Ac_{i_n,j}^n}  |\Jbb_{\hb_{a,b}}(x_n,y_n)|^s 
  \cdot \Phi_j \circ \hb_{a,b}(x_n,y_n) =0.
\end{align*}
This implies that for any $(a,b) \in \Ac_{i_n,j}^n$ we have $\Phi_j ( \hb_{a_1,b_1} \circ \cdots \circ \hb_{a_n,b_n}(x_n,y_n))=0$.

%Note  that  $\hb_{a_1,b_1} \circ \cdots \circ \hb_{a_n,b_n}(x_n,y_n)$  converges to a unique limit point as $n$ goes to infinity. 
By the density of periodic orbits from Proposition \ref{dense:periodic}  and by  the Earle--Hamilton fixed-point theorem together with Proposition \ref{hol:domain},  we deduce that,
for given  any $(x,y) \in I^2$ and $(x',y') \in I^2$, there exists an admissible  sequence   of  Ostrowski pairs of digits $(a_k,b_k)_k $ such  that 
$(x',y')=    \lim _ k \hb_{(a_1,b_1)} \circ  \cdots \circ \hb_{(a_k,b_k)} (x,y)$.
In particular,    any element  $(x',y')$ of $I^2$  can be reached via  $(x_n,y_n)$, i.e.,  there exists an admissible  sequence   of   Ostrowski pairs of digits $(a_k,b_k)_k $ (which depends possibly  on $n$)  such  that 
$(x',y')=   \lim _ k \hb_{(a_1,b_1)} \circ  \cdots \circ \hb_{(a_k,b_k)} (x_n,y_n)$.  By continuity of  the  function $\Phi$ and 
by  the admissibility condition (M2)  from  Definition \ref{markov:mayer},  together with   Proposition \ref{prop:Mayer}, this 
 implies that   %$ \tau_{a_1,b_1} \circ \cdots \circ \tau_{a_n,b_n}(i_n) $   takes both values $0,1$   infinitely often, and 
  $\Phi \equiv 0$.

 Hence for any non-trivial $\Phi \in K$, we see that $\Lt^n \Phi \in \mathrm{int}(K)$, i.e., $c_0 \leq \Lt^n \Phi \leq c_1$ for some real values $c_0, c_1 >0$ depending on $\Phi$. Together with (1), we conclude that $\Lt$ has a simple positive eigenvalue $\lambda_{1,0}$   associated with the positive eigenfunction $\Phi_{1,0}$ and the modulus of all other eigenvalues are strictly smaller than $\lambda_{1,0}$.
\end{proof}

This leads to the following  Kuzmin-type theorem  for the Ostrowski dynamical system, that is, the existence of an absolutely continuous invariant measure $\mu$ on $I^2$ that is  exponentially mixing. 

\begin{thm} \label{mixing}
The Ostrowski transformation admits an absolutely continuous invariant measure $\mu$ with  piecewise  holomorphic density. Moreover, the system $(I^2,S,\mu)$  is exponentially mixing in  $\kappa (\Bw)$.
\end{thm}

\begin{proof}
By Theorem \ref{spectrum0}, we have a spectral gap for $\Lt_{1,0}$ with a simple dominant eigenvalue $\lambda_{1,0}$ and corresponding positive eigenfunction $\Phi_{1,0}$. 
 Further, all other eigenvalues are strictly contained in an open disk of radius $\lambda_{1,0}$. Due to the specialisation (\ref{kappa:special}),  observe that $\kappa \Phi_{1,0}$ is an eigenfunction of $\Lc_{1,0}$ on $L^1(I^2)$ with the same eigenvalue $\lambda_{1,0}$. 
%Thus there exists a measure $\mu_{s,w}$ so that the adjoint operator $\Lc_{s,w}^*$ on $L^1(I)$ satisfies $\Lc_{s,w}^* \mu_{s,w}=\lambda_{s,w} \mu_{s,w}$. 
In particular, one  has  $\lambda_{1,0}=1$ by  Remark \ref{rem:int}. Thus there exists $\varphi_* \in \kappa (\Bw)$ satisfying $\Lc_{1,0} \varphi_* = \varphi_*$ (with $\varphi_*= \kappa \Phi_{1,0}$) and we get an invariant measure $\mu=\varphi_* m$ for the Ostrowski dynamical system, where $m$ stands for the Lebesgue measure on  $I^2$ (its explicit   expression, due to Ito \cite{Ito:86},  is given in Remark \ref{ito}).

We next prove the mixing property. Again by the gap in the spectrum of $\Lt_{1,0}$, we have that for all $\psi$ with $\psi=\kappa \Psi$ for some $\Psi \in \Bw(\Omega)$ 
\begin{align*}
\Lc_{1,0}^n \psi &= \Lc_{1,0}^n (\kappa \Psi)  =
 \kappa (\Lt_{1,0}^n \Psi) \\
&= \kappa ( \lambda_{1,0}^n \Pt_{1,0} \Psi + \Nt_{1,0}^n \Psi) \\
&= \lambda_{1,0}^n  \left(\int_{I^2} \psi \ dm \right) \varphi_*  +  O(\rho^n \| \Psi \| )  \numberthis \label{spectgap} 
\end{align*}
as $n$ goes to infinity. Here, we use the spectral decomposition $\Lt_{1,0}=\lambda_{1,0} \Pt_{1,0} +\Nt_{1,0}$. 

Finally consider the correlation function: for $\phi, \psi \in \kappa (\Bw)$ with $\psi=\kappa \Psi$ and $\phi=\kappa \Phi$, by (\ref{spectgap}), we have 
by noticing that 
$ \lambda_{1,0}=1$
\begin{align*}
\int_{I^2}  (\phi \circ S^n) \rdot \psi  \ \varphi_* d m &= \int_{I^2} \phi \cdot  \Lc^n_{1,0} (\psi \varphi_*) \  dm  \\
&= \int_{I^2}  \phi \ \varphi_* d m  \int_{I^2} \psi \ \varphi_* d m+ O(\rho^n \| \Phi \| \| \Psi \|), \numberthis \label{exp:mixing} 
\end{align*}
 as $n$ goes to infinity. This shows that the invariant measure $\mu$ for the Ostrowski system is  exponentially  mixing.
\end{proof}

\begin{rem} \label{ito}
We refer to Ito \cite{Ito:86}. Ito earlier showed that there is an invariant measure for the Ostrowski map and the  density  is explicitly given by 
$$ \varphi_*(x,y) =\left\{ 
\begin{array}{ll}
\frac{1}{2\log 2}  \frac{x+3}{(1+x)^2} \quad   \mbox{ on }  \Delta_0 \\
\frac{1}{2\log 2}  \frac{x+2}{(1+x)^2}  \quad \mbox{ on }  \Delta_1 .
\end{array}
\right.$$
One has $ \varphi_*=\kappa \Phi_{1,0}$. 
The proof is different from ours   in the sense that it is based on an explicit  realisation of the    natural  extension map. 

\end{rem}

\begin{rem}\label{rem:sg}

Write $s=\sigma+it$ and $w=\nu+ i \tau$ (with $i=\sqrt{-1}$). 
We remark that the bounded operator $\Lt_{\sigma,\nu}$ also satisfies the proof of Theorem \ref{spectrum0} and $\Lt_{s,w}$ depends analytically on $(s,w)$ with $(\sigma, \nu)$ near $(1, 0)$, i.e., there is a complex neighborhood $U$ of $(1,0)$ on which $\lambda_{s,w}$ and $\Phi_{s,w}$ depend analytically on $U$. Then by analytic perturbation theory (see Kato \cite{kato}), the spectral gap  property extends to $\Lt_{s,w}$ for all $(s,w) \in U$. 

We thus get, for $(s,w)$ close to $(1,0)$,  the spectral decomposition 
\begin{equation} \label{sp:dec}
\Lt_{s,w}=  \lambda_{s,w} \Pt_{s,w}+ \Nt_{s,w}, 
\end{equation}
where $\Pt_{s,w}$ is the projection onto the  $\lambda_{s,w}$-eigenspace %characterised by $$\Pt_{s,w}\Psi=  \Phi _{s,w} \int_{I^2} (\kappa \Psi)     \ dm  $$ 
and $\Nt_{s,w}$ corresponds to the remainder part of spectrum, i.e., this is a bounded operator with spectral radius $\rho$, where $\rho$ is the modulus of the subdominant eigenvalue. Hence, we summarize the perturbation of Theorem \ref{spectrum0} as follows.
 \end{rem}

We recall that in next statement  the  function $f \in \kappa ( \Bw(\Omega))$   which defines $\Lt_{s,w}$ is assumed to  be of moderate growth.
\begin{thm} \label{spectrum1} 
Let  $(\sigma, \nu)$ be  close to $(1,0)$.
\begin{enumerate}
\item The operator $\Lt_{s,w}$ on $\Bw(\Omega)$ is compact.
\item Further, there is a spectral gap. There is a positive eigenvalue $\lambda_{s,w}$ whose modulus is strictly larger than all other eigenvalues and moreover   the  corresponding eigenfunction $\Phi_{s,w}$ is positive.
\end{enumerate}
\end{thm}

\section{Central limit theorem for Diophantine parameters} \label{sec:clt}

Let $(I^2,S, \mu)$ be the Ostrowski dynamical system, where $\mu$ is the  invariant measure given as in Theorem \ref{mixing} and Remark \ref{ito}. We recall that for an observable $f$, the Birkhoff sum $S_nf$  is defined as
\[ S_n f= \sum_{k=0}^{n-1} f \circ S^k . \]
Assuming that $f$    has zero  integral  with respect to $\mu$, we say that the central limit theorem holds if the normalised sum $\frac{1}{\sqrt{n}} S_n f$  converges in law
as a random variable  to a normal distribution.

In this section, we establish  the central limit theorem  for  Birkhoff sums      for $f$ of moderate growth and  we  thus   observe  a  Gaussian behavior  for some Diophantine parameters. Note that this is a quite straightforward application of  Theorem \ref{spectrum1} with the  mixing property from Proposition \ref{mixing}.  More precisely, we  recall  by (\ref{exp:mixing}) that for $\phi, \psi$ with
$\phi=\kappa \Phi$   and $\psi=\kappa \Psi$ for some $\Phi, \Psi \in \Bw(\Omega)$, we have
\[ \int_{I^2} (\phi \circ S^n) \cdot \psi d\mu = \int_{I^2} \phi d\mu \int_{I^2} \psi d\mu+ O(\rho^n  \| \Phi \|\| \Psi \|) . \]
This not only shows  the mixing property in  $\kappa(\Bw (\Omega))$  with respect to the invariant measure $\mu$ but also  the exponential decay of correlation with  rate
$0<\rho<1$.  Note that we state  the central limit  property with the invariant  probability  mesure $\mu$ below but it also holds with the Lebesgue measure since  both measures  are equivalent. 

%Here we further establish a   central limit theorem for observables $f: I^2 \ra \Rb^+$  satisfying $f \circ \hb_{a,b}=O(\log a)$. 

\begin{thm} \label{clt}   
Let $f \in \kappa ( \Bw(\Omega))$  be a  non-negative real-valued function of moderate growth 
 that  has zero   integral $\int_{I^2} f d\mu$,   and which is not  of the form   $g-g \circ S$ for some $g$ of  $\kappa ( \Bw(\Omega))$. 
 
 Then there exists $\sigma>0$ such that $S_nf/ \sqrt{n}$ converges in law to the normal distribution, i.e.,
\[\lim_n \mu \left\{ (x,y) \in I :  \frac{1}{\sqrt{n}} S_n f(x,y) \leq z   \right\} \longrightarrow \frac{1}{\sigma \sqrt{2 \pi}} \int_{-\infty}^z e^{-t^2/ 2 \sigma^2} dt  \]
for all $z \in \Rb$ as $n$ goes to infinity.
%Here, $m$ stands for the Lebesgue measure on $I^2$.
\end{thm}

\begin{proof}
We follow the now classical approach  developed for instance in Broise \cite{Broise:96} or Sarig  \cite{SarigNotes} inspired by Nagaev's theory.

Consider $ \sum_{n  \geq 1 } \int_{I^2} (f \circ S^n) \cdot f d \mu$. Since  each integral in the summation   is a correlation function, the   exponential decay of (\ref{exp:mixing}) with  mixing rate $\rho<1$  implies that  the series  converges. Further we observe that 
\begin{align*}
\int_{I^2} (S_n f)^2 d \mu&=   \int_{I^2}\sum_{k_1, k_2=0}^{n-1} (f \circ S^{k_1}) \cdot (f \circ S^{k_2}) d\mu \\
%&=  \sum_{m=0}^{n-1} \left( \int_{I^2} f^2 d\mu+ 2 \sum_{k=1}^m \int_{I^2} (f \circ S^k) \rdot f d\mu    \right)\\
&=n  \int_{I^2} f^2 d\mu+ 2 \sum_{k=1}^{n-1} (n-k) \int_{I^2} (f \circ S^k) \rdot f d\mu   .
\end{align*}
One then checks that  $\sigma^2=\lim_{n \ra \infty} \frac{1}{n} \int_{I^2} (S_n f)^2 d\mu$  exists,    it
can be written as $\sigma^2=  \int_{I^2} f^2 d\mu+ 2 \sum_{k=1} ^{\infty}  \int_{I^2} (f \circ S^k) \rdot f d\mu    
$  (Green-Kubo formula),  and  that it
 is non-vanishing under the assumptions on $f$.

Next  we claim the normal distribution. Regarding $S_n f$ as a random variable, the idea is to express the moment generating function in terms of the transfer operator $\Lc_{1,w}$ with $w=it$ purely imaginary and observe the relation between derivatives of the dominant eigenvalue and $\sigma$. 
Observe that for $t \in \Rb$
$$
\Eb[e^{\frac{it}{\sqrt{n}} S_n f}] =  \int_{I^2} e^{\frac{it}{\sqrt{n}} S_n f} d\mu 
= \int_{I^2}  e^{\frac{it}{\sqrt{n}} (f \circ S^{n-1})} \cdots e^{\frac{it}{\sqrt{n}} f}   d\mu.$$
We recall that  $m$ stands for the Lebesgue measure on $I^2$ and that $\mu = \varphi_* m$.
With the  changes of variables $(x,y) \mapsto \hb_{a,b}(x,y)$, we have
\[ \int_{I^2}   e^{\frac{it}{\sqrt{n}} f} \varphi_* dm= \sum_{(a,b) \in \Ac} \int_{\hb_{a,b}(I^2)}  e^{ \frac{it}{\sqrt{n}} \cdot f \circ \hb_{a,b}}   |\Jbb_{\hb_{a,b}}(x,y)| \cdot  \varphi_* \circ \hb_{a,b} \, dm
=\int_{I ^2} \Lc_{1,\frac{it}{\sqrt{n}}}  ( \varphi_*)d m  , \]
and 
\begin{equation} \label{express:mgf}
\Eb[e^{\frac{it}{\sqrt{n}} S_n f}] = \int_{I^2} \Lc_{1,\frac{it}{\sqrt{n}}}^n ( \varphi_*)\, d m = \int_{I^2}\kappa (\Lt_{1,\frac{it}{\sqrt{n}}}^n ( \Phi_{1,0})) \,dm.
\end{equation}
Together with Theorem \ref{mixing}, finally we claim  the convergence  to $e^{-\frac{1}{2} \sigma^2 t^2}$. We recall \eqref{sp:dec} indeed that for a fixed $t$ and $n$ large enough, we have the spectral decomposition $\Lt_{1, \frac{it}{\sqrt n}}=  \lambda_{1, \frac{it}{\sqrt n}} \Pt_{1, \frac{it}{\sqrt n}}+ \Nt_{1, \frac{it}{\sqrt n}}$, which enables us to write the expression (\ref{express:mgf}) as
\begin{align*}
\Eb[e^{\frac{it}{\sqrt{n}} S_n f}] &=  
 \int_{I^2} \kappa (\Lt_{1,\frac{it}{\sqrt{n}}}^n (  \Phi_{1,0}))  \,dm \\
&= \int_{I ^2}  \lambda_{1,\frac{it}{\sqrt{n}}}^n \kappa( \Pt_{1,\frac{it}{\sqrt{n}}} ( \Phi_{1,0}))+ \kappa(\Nt_{1,\frac{it}{\sqrt{n}}}^n(  \Phi_{1,0})) \,dm \\
&= \int_{I^2}\left( 1+ \lambda_{1,0}^{(1)} \frac{t}{\sqrt{n}} + \frac{\lambda_{1,0}^{(2)}}{2} \rdot \frac{t^2}{n} + o\left(\frac{t^2}{n} \right) \right)^n \kappa( \Pt_{1,\frac{it}{\sqrt{n}}} ( \Phi_{1,0}) )+ \kappa(\Nt_{1,\frac{it}{\sqrt{n}}}^n(  \Phi_{1,0}))
\end{align*}
where $\lambda_{1,0}^{(m)}=\frac{\partial^m}{\partial w^m} \Big\vert_{w=0}\lambda_{1,w}$ for short. 
%\tcb{We use  $ \lim _n \int \kappa( \Pt_{1,\frac{it}{\sqrt{n}}} (   \Phi_{1,0})) dm  =1$.}

We now use the fact that we have the   following formulas for the  derivatives of  the  eigenvalue $\lambda_{1,w}$:
\begin{align*}
\lambda_{1,0}^{(1)}= \frac{\partial}{\partial w} \Big\vert_{w=0}\lambda_{1,w} &= \int_{I^2} f d\mu=0  \\
\lambda_{1,0}^{(2)}= \frac{\partial^2}{\partial w^2} \Big\vert_{w=0} \lambda_{1,w} &= -\sigma ^2.
\end{align*}
This  comes from the spectral gap and decay of correlations.

Then, we have the asymptotic expansion of $\lambda_{1, \frac{it}{\sqrt{n}}}$ when $n$ is large enough to conclude that $\frac{it}{\sqrt n}$ is close to $0$: \[ \lambda_{1, \frac{it}{\sqrt{n}}}=1-\frac{1}{2} \sigma^2 \left( \frac{t}{\sqrt{n}} \right)^2+ O\left( \frac{t}{\sqrt{n}} \right)^3 . \]
We then  deduce from   the  spectral decomposition   together with perturbation theory that 
$ \Eb[e^{\frac{it}{\sqrt{n}} S_n f}] $
 converges to $e^{-\frac{1}{2} \sigma^2 t^2}$.
%This finishes the proof for the normal distribution.
\end{proof}

In view of Theorem \ref{clt} and Remark \ref{rem:birk}, specialising the observable $f$ allows us to define  relevant quantities related to  inhomogeneous approximation,  and we have  a Gaussian distribution for such Diophantine parameters. 
Indeed, the  membership  functions  satisfy the assumptions of Theorem  \ref{clt}  (see e.g. \cite[Section 7]{Broise:96}).
We thus can consider  $\ib_{\leq N}(\lfloor \frac{1}{x} \rfloor) \ib_{ \lfloor \frac{1}{x} \rfloor -1}(\lfloor \frac{y}{x} \rfloor))$  for the set 
$E_N$  from (\ref{set:hausdorff}), or  else  $ \ib_{ k}(\lfloor \frac{y}{x} \rfloor))$, which yields to consider   $  \#\{ 1 \leq k \leq n :  b_k =b  \} $ (considered in \cite[Proposition 1]{ItoNakada}). 

 We focus  here on the quantity  $ y_n$, where $(x_n,y_n)= S^n(x,y)$ for all $n$.
We recall  from  (\ref{eq:Mn})  and  \S\ref{arith:iden} the  quantity    $M_n=\sum_{i=1}^n b_i q_{i-1} (-1)^{i-1}$.   By (\ref{eq:ybis}), one has 
 \[ \| y- M_n x\| =  \sum_{i=n+1}^\infty b_i |\theta_{i-1}|  = y_n |\theta_{n-1}|.\]
Hence the parameter $ y_n$ has a  Diophantine  expression  as $y_n=\frac{1}{|\theta_{n-1}|} \| M_nx- y\|$.
This dynamical parameter allows the 
definition  of  the set of badly approximable numbers \emph{with respect to the Ostrowski expansion} 
for $x \in [0,1) \backslash \Qb$ and $\varepsilon>0$, as 
\[ B_S(x, \varepsilon)= \left\{ y \in [0,1) : \liminf_{n \ra \infty}  \frac{1}{|\theta_{n-1}|}  \rdot  \| M_n x-y \|  \geq \varepsilon   \right\} . \]

Thus we have from Theorem \ref{clt}   a normal law for the quantity   $|\theta_{n-1}|^{-1}   \| M_n x-y \| $.
By Ito--Nakada \cite[Proposition 3]{ItoNakada}, one has 
%$$\lim \frac{ \#  \{ n : |\theta_{n-1}|^{-1}    \l M_n x-y \|     < z \}  - N  }{ \sigma \sqrt N} =  $$
for $z $ with $0 \leq z \leq 1$
\begin{align*}
&\lim_{N \rightarrow \infty} \frac{ \mbox{Card}  \{ n : 1 \leq n \leq N,  \  |\theta_{n-1}|^{-1}    \| M_n x-y \|     \leq z \}}  {N} \\
&\ \ \ \ \ \ =\frac{1}{2\log 2}  \left ((2+z)  \log 2 - (2-z) \log(2-z)z- z\log (1+z)  \right).
\end{align*}
This result is obtained by considering the Birkhoff sum associated to the membership  function   associated with the subset of $I^2$
defined by $0 \leq y \leq z$ by noticing that   $ \frac{1}{2\log 2}  ( (2+z)  \log 2 - (2-z) \log(2-z)z- z\log (1+z)  )= \int _{0 \leq y \leq z \leq 1} \varphi_* dm$ (see \cite{ItoNakada}).

We thus   deduce from Theorem \ref{clt}   the following  central limit  property   for the quantity   $y_n= |\theta_{n-1}|^{-1}     \| M_n x-y \| $:

\begin{align*}
& \mu \{ (x,y) \in I^2 :  \frac{1}{ \sigma \sqrt N}   \mbox{Card} \{ n :  1 \leq n \leq N, \  \frac{   \| M_n x-y \| }{ |\theta_{n-1}|} \leq z   \\ 
& \qquad  - N  \frac{1}{2\log 2}  ( (2+z)  \log 2 - (2-z) \log(2-z)z- z\log (1+z)  )   \}   \} \\
& \longrightarrow \frac{1}{\sigma \sqrt{2 \pi}} \int_{-\infty}^z e^{-t^2/ 2 \sigma^2} dt  .
\end{align*}

%\section{Hausdorff dimension in inhomogeneous approximation by the Ostrowski expansion} \label{sec:diophantine}

\section{Hausdorff dimension estimates} \label{sec:diophantine}
For classical continued fractions, the study  of  fractal sets    provided by     bounded digits arises   in a natural  way from  questions in Diophantine approximation. Let $A $ be a finite subset of positive integers  and  consider 
\[ \{ x \in [0,1): x=[0;a_1,a_2, \cdots ], \mbox{ with }    a_i \in A , \mbox{ for all }i \geq 1 \}.\] There have been powerful results concerning the implicit characterisation of the Hausdorff dimension of  this set  in terms of  (constrained)  transfer operators associated to the Gauss map  whose summand is constrained to $A$.  We refer to Hensley \cite{hens1}, Jenkinson--Gonzalez--Urba\'nski \cite{jenkinson},  Jenkinson--Pollicott \cite{jenkinson:pollicott}, or Das--Fishman--Simmons--Urba\'nski \cite{DFSU}. Roughly speaking, if the transfer  operator (having  the Jacobian as a potential)  has a spectral gap acting on a suitable function space and if  there is  a unique real  number $s$ for which the largest eigenvalue  is equal to 1, then $s$ is equal to the Hausdorff dimension of  this set. This kind of  precise description was first established by  Bowen \cite{bowen} in connection with the pressure function and by  Ruelle \cite{ruelle} via spectral invariants of transfer operators for conformal maps.

In this section, we obtain an analogous but weaker  statement for  the following  bounded-digit set for 
% that arises from the study of the  inhomogeneous Diophantine approximations provided 
Ostrowski expansions, where the associated dynamical system is non-conformal. 
We  focus here on the following set for a fixed $N \geq 2$: 
\[ E_N=\{ (x,y) \in [0,1)^2 :  1 \leq a_i \leq N \ \mbox{and }   b_i= a_i-1, \ \mbox{for all } i \geq 1 \}  \]
where $x=[0;a_1, a_2,\cdots]$ and $y$ is  determined by its  Ostrowski expansion with digits $(b_i)_i$ in base $x$
(see Fig. \ref{fig:3bis} for an illustration).
We are motivated on the one hand  by  Diophantine considerations concerning inhomogeneous  badly approximable numbers. On the other hand,  we  strongly use the  fact that we have an explicit description  of  the fundamental sets  $I_{a,b}$  from Proposition \ref{prop:diameter} in relation to the bounded distortion property of the Jacobian determinant, that we recall below.   We rely here on the fact that the map $S$ is a  skew-product of the Gauss map.

Notice that the set $E_N$ is simply the limit set 
\begin{equation} \label{EN:limitset}
E_N=\cap_{n \geq 1} E_N^{(n)} \mbox{ with } E_N^{(n)} :=  \cup_{(a,b) \in \Ac_N^n} \hb_{a,b}(I^2) 
\end{equation}
where $\Ac_N=\{ (a,b)  \in \Ac: a \leq N  , \ b=a-1\}$ is the set of indexes associated to $E_N$. 

%\begin{figure}[h] 
%   \centering
%   \includegraphics[width=2.5 in]{cell-1.png} 
%   \caption{The set $E_6^{(1)}$   associated to $\Ac_6$,
%    depicted in grey.}
%   \label{fig:3}
% \end{figure}

\begin{figure}[h] 
   \centering
   \includegraphics[width=2.5 in]{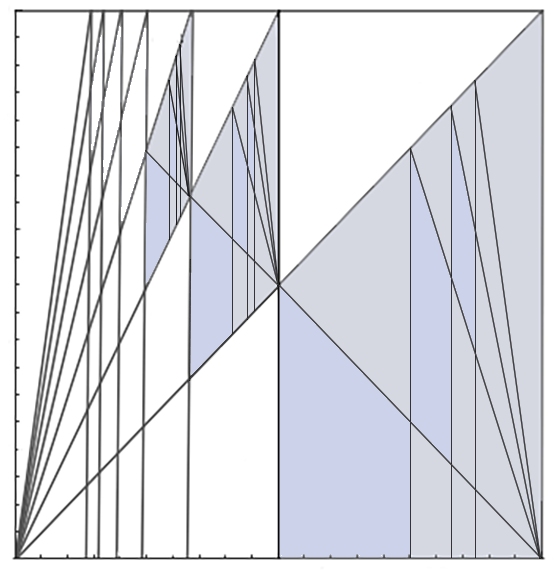} 
   \caption{The sets $E_3^{(1)}$ and $E_3^{(2)}$, depicted in grey and blue, respectively.}
   \label{fig:3bis}
\end{figure}

\begin{lem}[Bounded distortion] \label{distortion}
There exists a uniform constant $L>0$ so that for any  $(a,b) \in \Ac^n$ with $n \geq 1$ and   for any $(x_1,y_1)$ and $(x_2,y_2)$ in $I^2$, we have 
\[ \frac{1}{L} \leq  \left| \frac{ \Jbb_{\hb_{a,b}}(x_1,y_1)}{\Jbb_{\hb_{a,b}}(x_2,y_2)}    \right| \leq L , \]
where $ | \Jbb_{\hb_{a,b}}(x,y)|$ stands for  the Jacobian (determinant)   $ | \Jbb_{\hb_{a,b}}(x,y)|$ of  the inverse branch $\hb_{a,b}$.
\end{lem}

\begin{proof} According to  the notation of    Proposition \ref{prop:diameter}, one has 

\[ | \Jbb_{\hb_{a,b}}(x,y)| = \frac{1}{(q_n + xq_{n-1})^3 }\]   which does not depend on  the   skew coordinate $y$. 
We  deduce  the result  directly   from this expression,  by taking  e.g. $L=8$. \end{proof}

%Recall  from Definition \ref{bad:ost} and (\ref{set:hausdorff}) that badly approximable numbers $B_S(x,\epsilon)$ with respect to the Ostrowski map can be  characterised by  restrictions  on their digits. 

Note that there is  no Markov condition to consider here ($b$ is never equal to $a$).
In order to study the Hausdorff dimension of $E_N$, one has to  take suitable open covers. As classically done in the continued fraction setting, they will   be based on fundamental sets. Hence, for any $(a,b) \in \Ac_N^n$,  we consider the  fundamental   depth $n$ cylinder 
 $I_{a,b}$,    i.e.,   $\hb_{a_1,b_1} \circ \cdots \circ \hb_{a_n,b_n}(I^2)$,   which   is simply  a  trapezium 
(we use here the fact that  the digits $b_i$ are distinct  from $a_i$ for all  $i$). The fact that we have  explicit estimates
(by Proposition \ref{prop:diameter}) for  describing the  fundamental 
sets of depth $n$ is crucial.

Together with the  spectral properties from \S\ref{sec:spectral},  we aim to  get    Bowen--Ruelle type  estimates for the bounded-type set   $E_N$. Here, we consider $w=0$, so  we write $\Lc_{N,s}:=\Lc_{N,s,0}$, for short, for  the constrained operator  deduced   by the specialisation $\kappa$ from  the operator $\Lt_{N,s}=(\Lt_{0,(N,s)},\Lt_{1,(N,s)})$, whose summand is restricted to $\Ac_N$, i.e., for $i \in \{0,1\}$
$$
\Lt_{i,(N,s)} \Phi(x,y) = \sum_{j \in \{0,1\}} \sum_{(a,b) \in A_{N,i,j}} \frac{1}{(a+x)^{3s}} \cdot \Phi_j \left(\frac{1}{a+x}, \frac{b+y}{a+x}\right) 
$$
where  \[  \Ac_{N,i,j}=\{ (a,b) \in \Ac_N:  \hb_{a,b} ({\Delta_i})  \subset \Delta_j    \}  . \]
The constrained operator $\Lc_{N,s}$ has a spectral gap with the dominant eigenvalue $\lambda_{N,s}$ as the proof goes the same as Theorem \ref{spectrum1}. %This allows one  to  deduce the following    upper bound  for the Hausdorff dimension of $E_N$. %Consider the restricted Ostrowski map $S|_{E_N}$ and associated operator $\Lc_{N,s}$.

Recall the fact that  $$E_N= \cap_{n \geq 1} \cup_{(a,b) \in \Ac_N^n} I_{a,b}$$ from \eqref{EN:limitset}  and consider  open neighborhoods  containing
fundamental sets  $I_{a,b}$ %which are slightly larger 
which  form an open cover for $E_N$.  The approach is similar  to the  one classically 
performed for regular continued fractions, such as developed in Jarnik \cite{Jarnik}, Good \cite{Good}, or Fan--Liao--Wang--Wu \cite{FLWW:09}, for instance. 
We  now  obtain  the following upper bound.
\begin{prop} \label{up:hausdorff}
We have $\dim_H E_N \leq \min \{ s_2,  3s_1 \}, $ for  $s_1$   satisfying  $\lambda_{N,s_1}=1$ and $s_2$   satisfying    $(N+1)^{2 s_2} \lambda_{N,s_2}=1$.
\end{prop}

\begin{proof}
We fix $n$ and  some point $(x,y) \in E_N$.
First note that $$(\Lc^n_{N,s} \mathbf{1})(x,y)= \sum_{(a,b)\in \Ac_N^n}  |\Jbb_{a,b} (x,y)|^s $$ by taking the constant function $\mathbf{1}$.
By the spectral gap property (see Remark \ref{rem:sg}), one has $$|\Lc^n_{N,s} \mathbf{1} (x,y)-\lambda_{N,s}^n | \ll \lambda_{N,s}^n \rho^n$$ for some $\rho<1$. Also by Proposition \ref{prop:diameter}  (and using the notation of its statement),  one has 
$ \mathrm{diam}(I_{a,b})  \leq  \frac{3}{q_n +q_{n-1}}$ and by Lemma \ref{distortion}, we have 
\[  |\Jbb_{\hb_{a,b}}(x,y)| \geq L^{-1}  |\Jbb_{\hb_{a,b}}(0,0)|=L^{-1} \frac{1}{q_n } .\] 

Hence for any $(a,b) \in \Ac_N^n$ and $(x,y) \in I^2$,  this gives
\[ \mathrm{diam}(I_{a,b}) \leq    C_1 (N+1)^{2n}  |\Jbb_{\hb_{a,b}}(x,y)| \]
for some $C_1>0$  by noticing that all $a_i$'s satisfy  $a_i \leq N$, and then  $q_n \leq (N+1)^n$. Accordingly, we get 
\[   \sum_{(a,b)\in \Ac_N^n} \mathrm{diam}(I_{a,b})^s \leq   C_1^s (N+1)^{2ns}  \sum_{(a,b)\in \Ac_N^n}  |\Jbb_{a,b} (x,y)|^s \leq   C'_1  (N+1)^{2ns} \lambda_{N,s}^n  \]
for some $C'_1>0$. Taking the infimum over all open covers $\Oc$ whose elements have diameter at most $\delta$, we see that $ \inf \sum_{O_j \in \Oc} \mathrm{diam}(O_j)^s$ is bounded above by $C'_1 (N+1)^{2ns} \lambda_{N,s}^n$, where $n$ is chosen such that $\delta \geq \max_{\Ac_N^n} \mathrm{diam}(I_{a,b})$. This shows that if 
 $ (N+1)^{2s} \lambda_{N,s} <1$, then this infimum tends to 0 as $\delta$ goes to 0. Hence, we have $\dim_H(E_N) \leq s_2$ for  $s_2$ for which $(N+1)^{2s_2} \lambda_{N,s_2}=1$.

%\begin{rem} \label{rem:upperbound}

Similarly,  Lemma \ref{distortion}  together with Proposition \ref{prop:diameter}    gives,  for some $C_2,C'_2>0$,
\[   \sum_{(a,b)\in \Ac_N^n} \mathrm{diam}(I_{a,b})^{3s} \leq    C_2 ^{3s}   \sum_{(a,b)\in \Ac_N^n} \mathrm{Jac}(I_{a,b})^{s} \leq  (C_2')^s  \lambda_{N,s}^n,  \]
we also obtain $\dim_H E_N \leq 3s_1 $  for  $s_1$   satisfying $\lambda_{N,s_1}=1$.
\end{proof} 

  For the lower bound, we further observe:

\begin{prop} \label{low:hausdorff}
We have $\frac{3s_1}{2}  \leq dim_H E_N$  for  $s_1$   satisfying $\lambda_{N,s_1}=1$.
\end{prop}

\begin{proof}
We consider the Gibbs measure $\mu_{N,s}$ defined as the  probability  eigenmeasure of $\lambda_{N,s}$ for the dual  operator of  the constrained operator $\Lc_{N,s}$, whose existence is given   by   the analogue of Theorem \ref{spectrum1} for the constrained operator.
It satisfies, with some constant $C_3$,   for $(a,b) \in \Ac_N^n$ and for some fixed  $(x,y)\in E_N$, by also Lemma \ref{distortion}:
\begin{equation}\label{eq:jac3}
 \mu_{N,s}(I_{a,b}) =  \lambda_{N,s}^{-n}  \int_{I^2}(\Lc_{N,s}^n \mathbf{1}_{I_{a,b}}) d\mu_{N,s} \leq C_3 \lambda_{N,s}^{-n} |\Jbb_{\hb_{a,b}}(x,y)|^s.
\end{equation}
%hence  $$  \mu_{N,s}(I_{a,b})  \leq  C_3  \lambda_{N,s}^n \mathrm{diam}(I_{a,b})^{3s},$$ by  Proposition \ref{fundinterval}, for some constant $C_3>0$.
We now fix $s$  equal to $ s_1$  with $ \lambda_{N,s_1}=1$.

We use Notation \ref{nota:In} and Proposition  \ref{prop:diameter}. We recall   from \eqref{EN:limitset} that the set $$E_N=\cap_{n \geq 1} \cup_{(a,b) \in \Ac_N^n} I_{a,b}$$ is the intersection of the sequence of decreasing subsets $\cup _{(a,b) \in \Ac_N} I_n (a_1, \cdots, a_n)$ for $(a,b)=((a_1,a_1-1), \cdots, (a_n,a_n-1)) \in \Ac_N^n$.  For given  $(a_1, \cdots, a_n)$, we consider the natural decomposition (up to the boundaries of  the atoms) as follows (see Figure  \ref{fig:3ter} as  an illustration):  
\begin{equation} 
K_{n}  (a_1, \cdots, a_{n})= I_{n}  (a_1, \cdots, a_n) \backslash \bigcup_{  1 \leq k \leq N} I_{n+1}  (a_1, \cdots, a_{n},k) .
\end{equation}
%One has %$$K_{n}  (a_1, \cdots, a_{n}) =  \bigcup_{   k \geq N+1} I_{n+1}  (a_1, \cdots, a_{n},k)$$
 %$$K_{n}  (a_1, \cdots, a_{n}) =  \bigcup_{ (a',b' ) \in \Ac ^{n+1}, a'_1=a_1, \cdots, a'_{n-1}=a_{n-1} ,  a'_n \geq N+1,  b'_1=a_1-1, \cdots, b'_{n-1}=a_{n-1 }-1} I_{n+1}  (a',b')$$
 \begin{figure}[h] 
   \centering
   \includegraphics[width=2.5 in]{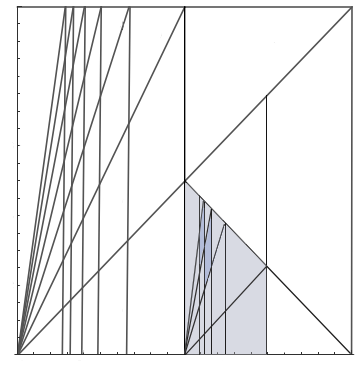} 
   \caption{The set $K_2(1,1)$, depicted in blue, inside $I_2(1,1)$ ($N=2$).}
   \label{fig:3ter}
\end{figure}
 
Then the set $K_{n}  (a_1, \cdots, a_{n}) $ is a union of fundamental  sets  $I_{a',b'}$  corresponding 
 to    elements  $(a',b' ) \in \Ac ^{n+1}$ satisfying $ a'_1=a_1, \cdots, a'_{n}=a_{n}$ and $a'_{n+1} \geq N+1$, also $ b'_1=a_1-1, \cdots, b'_{n}=a_{n }-1$ and $b'_{n+1} \leq a'_{n+1}$.  The  set  $ \bigcup_{  1 \leq k \leq N} I_{n+1}  (a_1, \cdots, a_{n},k)$    is a  connected component of $I_{n}  (a_1, \cdots, a_n)$ made of a finite union of  trapeziums, and  the trapezium  $K_{n}  (a_1, \cdots, a_{n})$
  is  taken away from $I_{n}  (a_1, \cdots, a_n)$ at level $n$    for defining $E_N$; we consider  it as a  ``hole''.
Notice that the hole is located on  the right of $I_{n}  (a_1, \cdots, a_n)$  if and  and only if  $n$ is odd. 

%According to Notation \ref{nota:In}, its  lower 
%vertices are $ D_{n+1}  (a_1, \cdots, a_n,N) $ and $A_{n} (a_1, \cdots, a_n) $  if $n$ is odd,  
%and  $ A_{n+1}  (a_1, \cdots, a_n,N) $ and $D_{n} (a_1, \cdots, a_n) $  if $n$ is even. 

%odd D  left.

 Let $B(\delta)$ be a closed ball of radius $\delta$ in $E_N$. 
Let $n$ be the largest integer such  that  $ B (\delta)\subset I_{n-1} (a_1,\cdots,a_{n-1})$, for some 
$(a_1,\cdots,a_{n-1})$.
The ball $B(\delta)$ thus intersects at least two sets of the form  $I_{n} (a_1, \cdots,  a_{n-1}, k+1)$   and $I_{n} (a_1, \cdots,  a_{n-1}, k)$ ($1\leq k \leq N-1$).
We assume w.l.o.g. that $ n$  is odd.   This implies 
that  the  hole  $K_{n-1}(a_1, \cdots,a_{n-1})$ is on the left  of $I_{n-1}  (a_1, \cdots, a_{n-1})$, and  thus 
that  the set  $I_{n} (a_1, \cdots,  a_{n-1}, k+1)$ is located on the left  of  $I_{n} (a_1, \cdots,  a_{n-1}, k)$.  Moreover,  the  hole  $K_{n}(a_1, \cdots,a_{n-1},k+1)$ is on the right of $I_{n} (a_1, \cdots,  a_{n-1}, k+1)$.
Since $B(\delta)$  intersects   $I_{n} (a_1, \cdots,  a_{n-1}, k+1)$   and $I_{n} (a_1, \cdots,  a_{n-1}, k)$,
then  $\delta$ is larger than  the difference of the lower vertices of the  set $ K_{n}(a_1, \cdots,a_{n-1},k+1)$.

It remains to  determine the lower vertices of  $ K_{n}(a_1, \cdots,a_{n-1},k+1)$. 
The fundamental set    $I_{n+1} (a_1, \cdots,  a_{n-1}, k+1,N)$  is located  just on the left   of  the hole  $K_{n}(a_1, \cdots,a_{n-1},k+1)$  within the fundamental set  $I_{n} (a_1, \cdots,  a_{n-1}, k+1)$.  According to Notation \ref{nota:In}, the   right lower  vertex  of  $I_{n+1}( a_1, \cdots, a_{n-1},k+1,N)$  is    $D_{n+1}( a_1, \cdots, a_{n-1},k+1,N)$. This gives that   the left lower vertex of   the hole $ K_{n}(a_1, \cdots,a_{n-1},k+1)$
is  $D_{n+1}( a_1, \cdots, a_{n-1},k+1,N)$. The  right lower vertex  of  the hole $K_{n}(a_1, \cdots,a_{n-1},k+1)$ is the right  lower vertex of  $I_n  (a_1, \cdots, a_{n-1},k+1)$,
 namely 
 $A_n(a_1, \cdots, a_{n-1},k+1)$. 
The lower vertices of  $K_{n}(a_1, \cdots,a_{n-1},k+1)$ are  thus  $D_{n+1}(a_1, \cdots, a_{n-1},k+1,N)$ and $A_{n}(a_1, \cdots, a_{n-1},k=1)$.
Hence $ \delta$ is  larger than  the absolute value of the  difference of their  abcisses, denoted by  $|-D_{n+1} (a_1, \cdots, a_{n-1} , k+1, N) +A_{n} (a_1, \cdots, a_{n-1} , k+1)|_x $.

Now, let us prove that  $\delta \geq   \frac{1} {(N+1)^3q_{n-1}^2 }.$ By   Proposition \ref{prop:diameter}, one has 
\begin{align*}
 |-D_{n+1} (a_1, \cdots, a_{n-1} , k+1, &N) +A_{n} (a_1, \cdots, a_{n-1} , k+1))|_x \\
 &= \frac{p_{n+1}  +p_{n} ( k+1)}{q_{n+1}  +q_{n} (k+1)} -\frac{p_{n} (k+1)}{q_{n} (k+1)}
 \end{align*}
by using  the following notation:
 $$p_{n+1}= p_{n+1}(a_1, \cdots, a_{n-1} , k+1,N), \  p_n(k+1)= p_n(a_1, \cdots, a_{n-1} , k+1),$$
$$p_{n-1}= p_{n-1}(a_1, \cdots, a_{n-1}),$$  with also  the  analogous  notation  for  the  $q_n$'s.
One has $q_n(k+1)= q_n (k) +q_{n-1}, $
and similarly for the  $p_n$'s.
This gives $$ \delta \geq \left| \frac{p_{n+1}  +p_{n} ( k+1)}{q_{n+1}  +q_{n} (k+1)} -\frac{p_{n} (k+1)}{q_{n} (k+1)}\right|= \frac{1}{(q_{n+1}  +q_{n} (k+1))q_{n} (k+1)}  \geq  
 \frac{1}{ (N+1) ^3 q_{n-1}^2} $$
by noticing that $q_n(k+1) \leq  (N+1) q_{n-1}$ and $q_{n+1} \leq (N^2+N+1)   q_{n-1} \leq (N+1)^2 q_{n-1}$.

Moreover,  there is a positive integer  $\ell $  such that  for all  $m$ and for  any  $\ell+m$-tuple $(a',b')$  of pairs  of digits  in $ \Ac_N^{\ell+m}$,  one has 
%$$\mathrm{diam}(I_{a,_1,\cdots,a_n(k), a_{n+1} , \cdots ,a_{n+\ell}}) \leq  2/ q _{n+\ell} \leq  \frac{1} {2(N+2)^2 q_{n-1} }   \leq  \delta$$ 
$$ \frac{1}{q_{m+\ell}(a',b')} \leq  \frac{1} {(N+1)^{3/2} q_{m-1}(a',b') }  \leq \delta^{1/2},$$
(here the dependence of the  $q_i$'s  with respect to the choice  of digits $(a',b')$ is  denoted as $q_{i}(a',b')$).
Hence,   by setting  $m=n$, Proposition \ref{prop:diameter} yields  that the  corresponding Jacobian  $ \Jbb_{\hb_{a',b'}} $ satisfies 
$$| \Jbb_{\hb_{a',b'}} |  \leq  \frac{1}{q_{n+\ell}^3(a',b')} \leq   \delta^{3/2}.$$

%$$\mathrm{Jac}(I_{a_1,\cdots,a_m, a_{m+1} , \cdots ,a_{m+\ell}}) ^{2/3}  \leq \frac{1}{q_{m+\ell}^2} \leq  \frac{1} {2(N+2)^3 q_{m-1}^2 }   \leq  \delta.$$ 

 Since  $ B (\delta)\subset I_{n-1} (a_1,\cdots,a_{n-1}) $,  $B(\delta)$ is covered by at most  $N^{\ell+1}$   cylinders  of depth $n+\ell$, and by (\ref{eq:jac3}), one has 
 $$\mu_{N,s_1}(B(\delta)) \leq   C_3  N^{\ell+1}     \delta ^{3s_1/2}.$$
Thus by mass distribution principle, we get $\dim_H(E_N) \geq 3s_1/2$.
% \begin{figure}[h] \label{fig:proof}
%   \centering
%   \includegraphics[width=2.5 in]{cell4.png} 
%   \caption{}
%   \label{fig:5}
%\end{figure}
\end{proof}

We end with one final observation by  discussing  a simple corollary following Hensley \cite{hens2}. Indeed, together with the spectral gap from Theorem \ref{spectrum1} and \ref{mixing}, the  following relation between the leading eigenvalues of $\Lc_{N,s}$ and $\Lc_{s}$, namely 
\begin{equation} \label{relation:eig}
\lambda_{N, 1-\frac{s}{N}}=\lambda_{1-\frac{s}{N}} + O(N^{-1}),
\end{equation}
is an almost straightforward adaptation of the arguments due to Hensley. Remark that Hensley \cite{hens2} obtain this more precisely with an accurate $N^{-1}$ order term and $O(N^{-2})$, but here we only consider a simple situation.

Looking at the Taylor expansion at $s=0$, we have 
\begin{equation} \label{taylor}
 \lambda_{1-\frac{s}{N}}=  1- \frac{s}{N} \cdot \frac{d}{ds} \lambda_s \Big\vert_{s=1}+O(N^{-2}) . 
\end{equation}
This shows that if one gets the first derivative of $\lambda_s$ of $\Lc_s$, then it is possible to calculate the eigenvalue  $\lambda_{N,s}$ of the constrained operator $\Lc_{N,s}$. %and hence the Hausdorff dimension of $E_N$ in view of Proposition \ref{up:hausdorff} and \ref{low:hausdorff}. 

  Consider  now the operator 
\[ \Ac \phi(x,y)=- \sum_{(a,b) \in \Ac} \frac{3 \log(a+x)}{(a+x)^3} \cdot \phi \left( \frac{1}{a+x}, \frac{b+y}{a+x}  \right) .\]
Note that this operator is simply defined by $\Ac \phi= \frac{d}{ds} \Lc_s \Big\vert_{s=1}$ and thus it satisfies
\begin{equation} \label{taylor2}
\Lc_s=\Lc_1+ (s-1)\Ac+O(|s-1|^2) 
\end{equation}
where $s$ is close to 1. Here by the decomposition theory of bounded operators (see Hensley \cite[Lemma 7]{hens2}), for the invariant density $\varphi_*$ from Remark \ref{ito}, we can write $\Ac \varphi_*= c \varphi_* + \eta$ for some $c\in \Cb$ with $\eta$ belonging to the image of $\Ac-cI$. Evaluating the projection operator, we have $\mathcal{P} \Ac \varphi_*=c \varphi_*$ since $\mathcal{P} \circ \eta =0$. Thus from integral representation for $\mathcal{P}$, we see  that  $c= \int_{I^2}\Ac \varphi_* dm$. We first claim that $c=\frac{d}{ds} \lambda_s |_{s=1}$. Set $z=s-1$ and $u=(I-\Lc_1)^{-1} \eta$, then (\ref{taylor2}) gives
\begin{align*} 
\Lc_s(\varphi_*+zu)&=\Lc_1(\varphi_*+zu)+z \Ac(\varphi_*+zu)+ O(|z|^2) \\
&=(\varphi_*+zu)+cz(\varphi_*+zu)+O(|z|^2)
\end{align*}
since we have $\lambda_1=1$. Hence for $|z|$  being sufficiently small, by perturbation theory (see e.g.  \cite[\S6]{hens2}), this  yields that $|\lambda_s - (1+cz)|\ll |z|^2$ and thus $\frac{d}{ds} \lambda_s |_{s=1}=c$.

Therefore we have 
\begin{align*} 
c=\frac{d}{ds} \lambda_s \Big\vert_{s=1}&= \int_{I^2} \Ac \varphi_*(x,y) d(x,y)  \\
%&= - \frac{1}{2 \log 2} \int_{\Delta_0} \sum_{(a,b)\in \Ac} \frac{\log(a+x)}{(a+x)^3} \cdot \frac{(3a+1)+3x}{(a+x)^2 ((a+1)+x)^2} d(x,y) \\
%& \ \ \ \ \ \ \ - \frac{1}{2 \log 2} \int_{\Delta_1} \sum_{(a,b)\in \Ac \atop a\neq b} \frac{\log(a+x)}{(a+x)^3} \cdot \frac{(2a+1)+2x}{(a+x)^2 ((a+1)+x)^2} d(x,y) \\
&=- \frac{1}{2 \log 2} \int_0^1 \sum_{a \geq 1} (a+1)x \cdot \frac{3\log(a+x)}{(a+x)^3} \cdot \frac{(3a+1)+3x}{(a+x)^2 ((a+1)+x)^2} dx \\
& \ \ \ \ \ \ \ + \frac{1}{2 \log 2} \int_0^1 \sum_{a \geq 1} a(1-x) \cdot \frac{3\log(a+x)}{(a+x)^3} \cdot \frac{(2a+1)+2x}{(a+x)^2 ((a+1)+x)^2} dx .
\end{align*}
Together with (\ref{relation:eig}) and (\ref{taylor}), we have $\lambda_{N, 1-\frac{s}{N}}=1-c \cdot \frac{s}{N}+O(N^{-2})$. 
%Thus we have: 
%\[ \dim_H(E_N)= 1-\frac{1}{N} \cdot \frac{N^2 O(N^{-2})+N^2-1}{cN}= 1-\frac{1}{c}\left(1-\frac{1}{N^2}+O(N^{-2})    \right)     .\]

\begin{rem}
Proposition \ref{up:hausdorff} and \ref{low:hausdorff} show that the Hausdorff dimension of $E_N$ can be estimated via the  bounded distortion property of  the Jacobian as an application of  the spectral analysis of the transfer operator due to nice properties of the Ostrowski dynamical system, despite the non-conformality. 
Though this is far from being  sufficient to have the same upper and lower bound. We  even expect both  bounds  not being  sharp,  the upper bound  being in particuar far from optimal, which we plan  to   investigate in the further work.
\end{rem}

\section{Appendix} \label{ap:ost}

\subsection{Proof of Proposition \ref{prop:diameter}}

Let $(a,b) =((a_1,b_1),  \ldots, (a_n,b_n)) \in \Ac^n$. We keep the notation of Proposition \ref{prop:diameter} for the definition of    $(p_k)_{ -1 \leq k \leq n}$ and $(q_k)_{ -1 \leq k \leq n}$.
 The    matrix  of the homography  $ \hb_{a,b}$  is
$M_{a_1, b_1}\cdots M_{a_n,b_n}$, where  $M_{a_i,b_i}:=\begin{bmatrix} 
0 & 0 & 1\\
0 &1 & b_i\\
1 & 0 & a_i
\end{bmatrix}.$  It  is thus  of the form 
$$\begin{bmatrix}
p_{n-1} & 0 & p_n\\
\alpha_{n-1} & 1 & \alpha_n\\
q_{n-1} & 0 & q_n
\end{bmatrix}$$
with 
\begin{equation}\label{eq:alpha3}
\alpha_n= \sum _{i=1}^n  b_i (-1)^{i} (q_n p_{i-1} - p_n q_{i-1})=  \sum _{i=1}^n  b_i  |q_n p_{i-1} - p_n q_{i-1}|.
\end{equation}
Note that  the expressions  for $\alpha_n$ can be obtained by induction, with the induction relation being  $ \alpha_{k+1}= \alpha_{k-1} + b_{ k+1} + \alpha_k a_{ k+1}$,  by observing  that the  sign of 
 $q_np_{i-1}-p_nq_{i-1}$ is  given by  $(-1)^{i}$ ($1 \leq i \leq n-1$)  and thus does not depend on $n$.
 We  directly deduce   from this expression  the   first statement on the Jacobian.
 
The set $I_{a,b}$   is either equal to  $ \hb_{a,b}( I^2)$ or $ \hb_{a,b}( \Delta_0)$.  
According to Notation \ref{not:h} and from   the expression   $ \hb_{a,b} (x,y)= (\frac{p_{n-1} x + p_n}{q_{n-1} x + q_n}, \frac{\alpha_{n-1} x +y+ \alpha_n}{q_{n-1} x + q_n}),$
one deduces that the  vertices of $I_{a,b}$    belong to the set  $A_n= \hb_{a,b}(0,0)$,  $B_n= \hb_{a,b}(0,1)$,  $C_n=\hb_{a,b}(1,1)$, and $D_n=\hb_{a,b}(1,0)$,
with
\begin{align*}
A_n&=\left( \frac{p_n}{q_n}, \frac{\alpha_n}{q_n} \right), \ 
B_n=\left( \frac{p_n}{q_n},  \frac{\alpha_n +1 }{q_n} \right),\\
   C_n&=\left(\frac{p_n+p_{n-1}}{q_n+q_{n-1}},  \frac{\alpha_n+ \alpha_{n-1}+1}{q_n+q_{n-1}} \right), \ 
D_n=  \left(\frac{p_n+p_{n-1}}{q_n+q_{n-1}},  \frac{\alpha_n+ \alpha_{n-1}}{q_n+q_{n-1}}  \right)  .
\end{align*}
 Let $K_n$ stand  for   the maximum of the absolute values of   the cofactors of  the matrix  of $ \hb_{a,b}$.  One checks  that  the diameter  of  $I_{a,b}$
 is smaller than or equal to  $3 K_n/q_n^2$.  
 One has  $$K_n=\max ( |-p_n \alpha_{n-1} + \alpha_n p_{n-1}|, |-q_n \alpha_{n-1} + \alpha_n q_{n-1}|, 1, p_n, q_n, p_{n-1}, q_{n-1}).$$
Moreover, one deduces from   (\ref{eq:alpha3}) 
%$$-q_{n+1} \alpha_n+q_n \alpha_{n+1}=  q_n \alpha_{n-1} - \alpha_n q_{n-1} + q_n b_{n+1},$$
that,   for any positive $n$, 
\begin{equation}\label{eq:alpha2} -q_n \alpha_{n-1} + \alpha_n q_{n-1}=    \sum _{i=1}^n  b_i (-1)^{n+i}  q_{i-1},
\end{equation}
and similarly $ -p_n \alpha_{n-1} + \alpha_n p_{n-1}=    \sum _{i=1}^n  b_i (-1)^{n+i}  p_{i-1}.$
This implies that       $K_n =q_n$ and  thus that the diameter of $I_{a,b}$ is   in $O(1/q_n)$.

Now we assume that $(a,b) \in  \Ac_N^n$. We will get more precise estimates in this case. We have $I_{a,b}= \hb_{a,b}( I^2)$.  In particular $I_{a,b}$ is  quadrangular (and not triangular).
The measure of $I_{a,b}$ thus satisfies
$$
\int _{I^2}  \left(\frac{1}{q_{n} +x q_{n-1}}\right)^3 dxdy =  \int _{[0,1]}  \left(\frac{1}{q_{n} +x q_{n-1}}\right)^3 dx =\frac{q_{n-1}+2q_n}{2q_n^2 (q_n+q_{n-1})^2}
.$$

We recall that the four vertices of $I_{a,b}$   are denoted by $A_n= \hb_{a,b}(0,0)$,  $B_n= \hb_{a,b}(0,1)$,  $C_n=\hb_{a,b}(1,1)$, and $D_n=\hb_{a,b}(1,0)$.
The trapezium  $I_{a,b}$ is  inscribed  in  a rectangle with  parallel  vertical sides      of  width $\frac{1}{q_n (q_n +q_{n-1})}$.
One has $B_n$  above $A_n$ and  $C_n$ above $D_n$.    Its two sides have respective lengths $\frac{1}{q_n}$ and $\frac{1}{q_n + q_{n-1}}$.
Hence the diameter  of this quadrangular is bounded  below by $\frac{1}{q_n+q_{n-1}}$.

  One has 
$$\sum_{i=1} ^{n}  (-1) ^{i} a_i q_{i -1}=(-1)^{n} (q_{n} -q_{n-1} )- q_0+q_{-1}=(-1)^{n} (q_{n} -q_{n-1} )-1,$$
which gives 
\begin{equation}\label{eq:sum}
\sum_{i=1} ^{n} (-1) ^{i} (a_i -1) q_{i -1}=(-1)^{n} (q_{n} -2q_{n-1})  + \sum_{i=1}^{n-2} (-1)^i  q_{i} .
\end{equation}

Hence by   \eqref{eq:alpha2} and  \eqref{eq:sum}, this gives
$$
-q_n \alpha_{n-1} + \alpha_n q_{n-1}=  q_{n} -2q_{n-1}  + (-1)^{n} \sum_{i=1}^{n-2} (-1)^i  q_{i} .
$$
For two points $A$ and  $B$,  the notation    $(A-B)_y$ stands   for the  difference  between their coordinates. We have
$$(C_n-A_n)_y= \frac{-\alpha_n q_{n-1}+ \alpha_{n-1}q_n + q_n }{q_n (q_n +q_{n-1})}.$$
Hence by   (\ref{eq:alpha2}) and (\ref{eq:sum}), one gets 
$$(C_n-A_n)_y= \frac{ \sum _{i=1}^n  b_i (-1)^{n+i}  q_{i-1}+ q_n }{q_n (q_n +q_{n-1})}
= \frac{2q_n  -2 q_{n-1} +(-1)^n   \sum_{i=1}^{n-2}  (-1)^i q_i} {q_n(q_n + q_{ n-1})}   ,$$
which implies  
$$ 0< (C_n-A_n)_y < \frac{2}{q_n+q_{n-1}}.$$

Similarly, one has  

$$(B_n-D_n)_y=\frac{\alpha_n q_{n-1}- \alpha_{n-1}q_n +q_n +q_{n-1}}{q_n (q_n +q_{n-1})},$$
which gives
$$0< (B_n-D_n)_y=   \frac{ 2q_n -q_{n-1}+  (-1)^{n} \sum_{i=1}^{n-2}  (-1)^i q_i} {q_n(q_n + q_{ n-1})} < \frac{2}{q_n+q_{n-1}} .$$

%\tcb{In other words,  we could  used that the diameter is  bounded above by   $2 K_n/q_n^2$,
% where   $K_n$ the maximum  of the absolute values of the  minors of order 2 of the matrix   (\ref{mat:form}).}

%$$(A_n-C_n)_y=  \frac{q_{n-1} a_n -2q_{n-2}+(-1)^{n-1}  \sum_{i=1}^{n-3} (-1)^i q_i} {q_n(q_n + q_{ n-1})}.$$

%$$(B_n-D_n)_y=   \frac{ q_n +q_{n-1}+(q_{n} -2q_{n-1})+  (-1)^{n} \sum_{i=1}^{n-2}  (-1)^i q_i} {q_n(q_n + q_{ n-1})} < 1/q_n .$$

\subsection{Proof of Proposition \ref{ost}}

\begin{proof}[Proof of Proposition \ref{ost}] 
We follow here  mainly  the proof given in  \cite{Bourla}.

We   first  state some useful  identities. Recall that 
$|\theta_{k-1}|=a_{k+1} |\theta_{k}| +  |\theta_{k+1}| $   ($k \geq 0$).
Further, we have  $\theta_{-1}=1$ and $\theta_0=x$.  We also  recall that the series $\sum_{k=0}^{\infty} a_k  |\theta_k|$ is convergent as established in \S\ref{arith:iden},
and that  $\sum_{k=n}^{\infty} a_k |\theta_{k-1}|= |\theta_{n-2}|+ |\theta_{n-1}| $ holds for all positive $n$ (by (\ref{eq:reste})).
By using telescoping sums, one gets  
$$1+ x=   \sum_{k =1}^{\infty}  a_k | \theta_{k-1}|  , \quad 1- x = \sum _{k=1}^{\infty}  a_k \theta_{k-1}$$
which gives 
\begin{equation}\label{eq:reste2}
x= \sum_{k =1}^{\infty} a_{2k} |\theta_{2k-1}|, \quad  1=  \sum_{k =0}^{\infty} a_{2k+1} |\theta_{2k}|.
\end{equation} 
%Moreover, one gets
%\begin{equation}\label{eq:reste}
%\sum_{k=n}^{\infty} b_k |\theta_{k-1}| \leq  |\theta_{n-2}|+ |\theta_{n-1}| .
%\end{equation} 

We first show that the sequence of digits $(a_i,b_i)_{i \geq 1}$ produced by the  Ostrowski map applied to $(x,y)$
satisfies  the admissibility  conditions (\ref{ost:1}--\ref{ost:3}).
Recall that
$a_i=\lfloor \frac{1}{x_{i-1}} \rfloor$ and $b_i= \lfloor \frac{y_{i-1}}{x_{i-1}} \rfloor$ ($i \geq 1$).
Since  $y <1$, we have $ \lfloor \frac{y_{i-1}}{x_{i-1}} \rfloor \leq  \lfloor \frac{1}{x_{i-1}} \rfloor$, which  yields (\ref{ost:1}).
Assume that $a_i=b_i$.  One has   $x_i= \frac{1}{x_{i-1}}- a_i$, and  $y_i= \frac{x_{i-1}}{y_{i-1}} - b_i$.
One has $y_i < x_i$,   since  $\frac{y_{i-1}}{x_{i-1}}  <  \frac{1}{x_{i-1}}$ and $a_i=b_i$.
This implies that $b_{i+1}=  \lfloor \frac{y_{i}}{x_{i}} \rfloor=0$, which  gives (\ref{ost:2}).
Finally we claim    (\ref{ost:3}). 
One has  $ y= \sum_{i=1}^\infty b_i |\theta_{i-1}| $. Assume that  $a_i = b_i$ for all   odd  integer greater than some $i_0\geq 1$. In particular,  by   Condition (\ref{ost:2}),  
$b_i=0$ for all  even index larger than $i_0$.  Applying Ostrowski's map $S$  to  $(x_{i_0-1}, y_{i_0-1})$  gives the  expansion  $(a_i,b_i)_{i \geq i_0}$. This implies by (\ref{eq:reste2}) that $1= \sum_{i\geq i_0}^\infty a_i |\theta_{i-1}|=  y_{i_0} <1,$ which yields the desired contradiction. The same reasoning applies for the case of even indices. 
%One has  $ \sum_{i\geq i_0}^\infty a_i |\theta_{i-1}|= |\theta_{i_0-1}| + |\theta_{i_0-2}| .$

We now  prove the uniqueness of  the expansion.
Let us fix $y \in [0,1)$. Let $(b_i)_{i \geq 1}$ and $(b'_i)_{i \geq 1}$  be such that 
\[ y= \sum_{i=1}^\infty b_i |\theta_{i-1}| =  \sum_{i=1}^\infty b'_i |\theta_{i-1}|\]
where 
$0 \leq b_i \leq a_i$ for all $i \geq 1$, $b_{i+1}=0$ if $a_i=b_i $,  $b_i \neq a_i$ for infinitely many even  and odd  integers,
and the same assumptions for $(b'_i)_{i \geq 1}$.
Suppose that $(b_i)_{i } \neq (b'_i)_{i }$.
Let $k$ be the smallest  positive  integer $i$   such that $b_i \neq b'_i$. We assume without loss of generality that $b_k > b'_k$. Then one has   $$ \sum_{i=k+1}^\infty ( b'_i -b_i) |\theta_{i-1}|= \sum_{i=1}^k (b_i- b'_i)  |\theta_{i-1}| = (b_k- b'_k)  |\theta_{k-1}|   \geq |\theta_{k-1}|.$$

\begin{itemize}
\item First assume that $(b'_{k+1} - b_{k+1} ) \leq  a_{k+1} -1 $. By (\ref{eq:reste}) and since  $b'_i \neq a'_i$ for infinitely many $i$, one has $$ \sum_{i \geq k+2}  (b'_i - b_i)   |\theta_{i-1}|  <  | \theta_{k} |+ |\theta_{k+1}|.$$ Hence we get
\[ |\theta_{k-1}| \leq  \sum_{i \geq k+1}  (b'_i - b_i)   |\theta_{i-1}|  <  a_{k+1}  | \theta_{k} |+  |\theta_{k+1}| =|\theta_{k-1}| ,\]
which  gives  a contradiction. 

\item We  now assume   that $b'_{k+1} - b_{k+1}  = a_{k+1}$.  This implies that  $b_{k+1}=0$, $b'_{k+1}= a_{ k+1}$ and  consequently 
$b'_{k+2}=0$.  
Let $\ell $ be the smallest  positive  integer  such that $b'_{k+1+ 2\ell } - b_{k+1 +2 \ell  }   < a_{k+1+ 2 \ell}$ (such an index exists since  $b'_i \neq a_i$ for infinitely many even  and odd  integers). 
 One has  again by  (\ref{eq:reste}), and since  $b'_i \neq a_i$ for infinitely many $i$,
\[    \sum_{i \geq k+2  \ell + 2}  (b'_i - b_i)   |\theta_{i-1}|  <    | \theta_{k+2\ell} |.  \] 
Consequently,  one has 
\begin{align*}
 |\theta_{k-1}| \leq \sum_{i \geq k+1}  (b'_i - b_i)   |\theta_{i-1}|&=  \sum_{i=0}^{\ell-1}  a_{k+1+2i} |\theta_{k+2i}|+ \sum_{i \geq k+2\ell+1 }  (b'_i - b_i)   |\theta_{i-1}|\\
&<  \sum_{i=1}^{\ell-1}  a_{k+1+2i} |\theta_{k+2i}|+  | \theta_{k+2 \ell} | \leq |\theta_{k-1}|,
\end{align*}
which yields  the desired  contradiction, by noticing that
\[ \sum_{i=0}^{\ell-1}  a_{k+1+2i} |\theta_{k+2i}|+ |\theta_{k+2\ell}|= |\theta_{k-1}|. \]
\end{itemize}
\end{proof}

\hide{Let us  distinguish  the  two  cases  $b_{k+2} \geq 1$ and  $b_{k+2} =0$. We   first  assume that  $b_{k+2} \geq 1$. Then one has, again by  (\ref{eq:reste}),
\[ \sum_{i \geq k+2}  (b'_i - b_i)   |\theta_{i-1}|   =-b_{k+2}   |\theta_{k+1}|  +  \sum_{i \geq k+3}  (b'_i - b_i)   |\theta_{i-1}|  \leq  (-b_{k+2} +1)   | \theta_{k+1} |+ |\theta_{k+2}|  \] 
which this gives  again a contradiction  by
  $$ |\theta_{k-1}| \leq \sum_{i \geq k+1}  (b'_i - b_i)   |\theta_{i-1}| < a_{k+1}  | \theta_{k} |+  |\theta_{k+1}|=|\theta_{k-1}|.$$}
  \hide{Lastly,    assume that  $b_{k+2} =0$.   One has,   again by  (\ref{eq:reste}), and since  $b'_i \neq a_i$ for infinitely many $i$,
\[ \sum_{i \geq k+2}  (b'_i - b_i)   |\theta_{i-1}|   =  
 \sum_{i \geq k+3}  (b'_i - b_i)   |\theta_{i-1}|  <    | \theta_{k+2} |+ |\theta_{k+1}|  \] 
which yields  again the same  contradiction.}

\bibliographystyle{abbrv}
\bibliography{SkewBib}

\end{document}